\documentclass[12pt]{article}
\input epsf.tex
\usepackage{graphicx}
\usepackage{amsmath,amsthm,amsfonts,amscd,amssymb,comment,eucal,latexsym,mathrsfs}
\usepackage{stmaryrd}
\usepackage[all]{xy}

\usepackage{epsfig}

\usepackage[all]{xy}
\xyoption{poly}
\usepackage{fancyhdr}
\usepackage{wrapfig}
\usepackage{epsfig}
\usepackage{tikz-cd}

\theoremstyle{plain}
\newtheorem{thm}{Theorem}[section]
\newtheorem{prop}[thm]{Proposition}
\newtheorem{lem}[thm]{Lemma}
\newtheorem{cor}[thm]{Corollary}

\theoremstyle{definition}
\newtheorem{defn}{Definition}
\newtheorem{notation}[thm]{Notation}
\theoremstyle{remark}
\newtheorem{remark}{Remark}

\advance \topmargin by -\headheight
\advance \topmargin by -\headsep
\evensidemargin \oddsidemargin
\newcommand{\hcA}{{\hat{\mathcal{A}}}}

\def\A{{\cal A}}

\def\cA{{\cal A}}

\def\cB{{\mathcal B}}

\def\C{{\mathbb C}}

\def\cc{{\curvearrowright}}

\def\F{{\mathbb F}}

\def\N{{\mathbb N}}

\def\Q{{\mathbb{Q}}}

\def\R{{\mathbb R}}

\def\P{{\cal P}}

\def\cP{{\mathcal P}}

\def\R{{\mathbb R}}

\def\T{{\mathbb T}}

\def\chix{{\raise.5ex\hbox{$\chi$}}}

\def\Z{{\mathbb Z}}

\newcommand\Ad{\operatorname{Ad}}

\newcommand\Ann{\operatorname{Ann}}
\newcommand\Aut{\operatorname{Aut}}

\newcommand\Char{\operatorname{Char}}

\newcommand\ep{\varepsilon}

\newcommand\Hom{\operatorname{Hom}}

\newcommand\Inn{\operatorname{Inn}}

\newcommand\Image{\operatorname{Image}}
\newcommand\IRS{\operatorname{IRS}}

\newcommand\maxorder{{\operatorname{Maxorder}}}
\newcommand\Ker{\operatorname{Ker}}

\newcommand\Sch{\operatorname{Sch}}
\newcommand\CRS{\operatorname{CRS}}

\newcommand\Sub{\operatorname{Sub}}
\newcommand\Subp{\operatorname{Sub_{\F_q}}}

\newcommand\RS{\operatorname{RS}}

\newcommand{\resto}{\upharpoonright}

\begin{document}
\title{Characteristic random subgroups of  geometric  groups and free  abelian groups of infinite rank}
\author{Lewis Bowen\footnote{email:lpbowen@math.utexas.edu, NSF grant DMS-0968762 and NSF CAREER Award DMS-0954606} \\ University of Texas \\ $\ $ \\ Rostislav Grigorchuk \footnote{email:grigorch@math.tamu.edu, NSF grant  DMS-1207699} \\ Texas A\&M University\\ $\ $ \\ Rostyslav Kravchenko \footnote{email:rkchenko@gmail.com} \\ University of Chicago}

\maketitle

\begin{abstract}
%Let $G$ be a discrete group. An {\em invariant random subgroup} (IRS) for $G$ is a random subgroup whose law is conjugation-invariant. A {\em strongly-invariant random subgroup} (CRS) for $G$ is a random subgroup whose law is invariant under all automorphisms of $G$. The law of an IRS (CRS) is also called an IRS (CRS). An IRS is {\em continuous} if it is purely non-atomic. 
We show that if $G$ is a non-elementary word hyperbolic group, mapping class group of a hyperbolic surface or the outer automorphism group of a nonabelian free group then $G$ has $2^{\aleph_0}$ many continuous ergodic invariant random subgroups. If $G$ is a nonabelian free group then $G$ has $2^{\aleph_0}$ many continuous $G$-ergodic characteristic random subgroups. We also provide a complete classification of characteristic random subgroups of free abelian groups of countably infinite rank and elementary $p$-groups of countably infinite rank. 
\end{abstract}


\tableofcontents

\section{Introduction}
The goal of this paper is to classify characteristic random subgroups of the free abelian group of infinite rank and of the infinite dimensional torus and use these random subgroups to construct large families of continuous invariant random subgroups in some important non-commutative groups of a geometric nature. The topics of ergodicity and weak-mixing are also included in our research. To achieve this goal we will make use of Pontyagin duality, ergodic theory of groups actions, geometric group theory and an important result of Adyan \cite{A} on the existence of independent identities (laws) in a free group of rank $\ge 2$ (the latter is based on the solution of Burnside's problem). Let us explain this in more detail.

Let $G$ be a locally compact Hausdorff metrizable topological group with a countable basis of open sets. The space $\Sub(G)$ of all closed subgroups of $G$ admits a natural topology (called the Chabauty topology) (see \cite[Chapter E.1]{BP92}). In this topology $\Sub(G)$ is a compact metrizable topological space with a countable basis of open sets. One can characterize the convergence as follows. $H_n\to H$ if and only if two conditions hold. First, if $g_{i}\in H_{n_i}$ converges to $g$, then $g\in H$, and second, if $g\in H$ then there are $g_n\in H_n$ such that $g_n\to g$. It is easy to see that the conjugation action of $G$ on $\Sub(G)$ is continuous. 

%For example, in the special case in which $G$ is countable with the discrete topology, the topology on $\Sub(G)$ is the topology of pointwise convergence. Then $\Sub(G)$ is a totally disconnected compact metrizable set, which can be naturally identified with a closed subset of $\{0,1\}^G$ with the Tychonov topology. The identification is by the rule that $H\leq G$ is mapped to its characteristic function $\chi_H:G\to\{0,1\}$. In fact, for any locally compact group $G$, the space $\Sub(G)$ is compact (see for instance \cite{BP92}). 

In general, by a {\em random subgroup} (RS) of $G$ we mean a random variable which takes values in $\Sub(G)$. We will often identify it with its law, which is a Borel probability measure on $\Sub(G)$.

An {\em invariant random subgroup} (IRS) of $G$ can be interpreted as a Borel probability measure $\mu$ on $\Sub(G)$ that is invariant under the conjugation action of $G$ (i.e. the action of the group of inner automorphisms $\Inn(G)$). We also say that a random subgroup $K\le G$ is an IRS if its law is conjugation-invariant. 

For example, if $N\vartriangleleft G$ is normal then the Dirac mass $\delta_N$ is an IRS. Also, if $\Gamma\le G$ is a lattice in a locally compact group $G$, then there is an IRS $\mu_\Gamma$ of $G$ supported on the conjugacy class of $\Gamma$ ($\mu_\Gamma$ is the stabilizer of a Haar-random element of $G/\Gamma$ \cite{AB+11}).  
Thus IRS's provide a common generalization of normal subgroups and lattices. More generally, if for some $H\le G$, the normalizer $N_G(H)$ of $H$ in $G$ is a lattice in $G$ then there is a naturally associated IRS $\mu_H$ supported on the conjugacy class of $H$. To see this, let $gN_G(H) \in G/N_G(H)$ be a random coset with law equal to the Haar measure. Then $gHg^{-1}$ is an invariant random subgroup. If $G$ is countable then any IRS of this form is purely atomic. $\mu\in\IRS(G)$ is called {\em continuous} if it has no atoms. It is called {\it ergodic} if the dynamical system $(G,\Sub(G),\mu)$ is ergodic. 

There has been a recent increase in the study of IRS's (see for example \cite{ABBGNRS12, BGK13} and the references therein). 
Observe that groups with $2^{\aleph_0}$ many continuous ergodic IRS's include nonabelian free groups \cite{Bo12}, classical lamplighter  groups \cite{BGK13}, the group of finitely-supported permutations of the natural numbers \cite{Ve11}. It is known that branch and weakly branch groups  have at least one continuous ergodic IRS \cite{BG00,Gr11}. There are also results in the opposite direction: groups without any continuous ergodic IRS's include lattices in simple higher rank Lie groups \cite{SZ94}, Higman-Thompson groups \cite{DM}, most special linear groups over countable fields \cite{PT}, commensurators of irreducible lattices in most higher rank semisimple Lie groups \cite{CP} and irreducible lattices in semisimple property (T) higher rank Lie groups \cite{Cr}.

We will denote by $\IRS(G)$ the set of all IRS's of $G$. Thus is a weak*-closed convex subset of the simplex $\RS(G):=\P(\Sub(G))$ of all probability measures on $\Sub(G)$. $\IRS(G)$ consists of all $\Inn(G)$-measures in $\RS(G)$, where $\Inn(G)$ is the group of inner automorphisms of $G$. Another important notion is that of characteristic random subgroup (abbreviated as CRS). Recall that a {\em characteristic} subgroup $K\le G$ is a subgroup which is $\Aut(G)$-invariant. For example, the torsion part of a countable abelian group is characteristic.
\begin{defn}
We define $\CRS(G)$ to be the closed convex subset of $\IRS(G)$ consisting of $\Aut(G)$-invariant measures, where $\Aut(G)$ is the group of all continuous automorphisms of $G$. Elements of $\CRS(G)$ are called characteristic random subgroups of $G$.
\end{defn}

 If $\Phi\leq \Aut(G)$ is a subgroup, one can consider a set $\RS_\Phi(G)$ of $\Phi$-invariant probability measures on $\Sub(G)$. Such a situation arises for instance when $N\unlhd G$ is a normal subgroup and we are interested in the study of random subgroups in $N$ invariant with respect to conjugation by elements of $G$. In this case $\Phi\leq \Aut(N)$ equals the image of $\Inn(G)$ in $\Aut(N)$. An example of this sort is considered in \cite{BGK13}.
 
  The spaces $\IRS(G)$, $\CRS(G)$, $\RS_\Phi(G)$ are all Choquet simplexes (\cite[Section 12]{Ph01}). Their extreme points are indecomposable measures, i.e. measures $\mu$ that cannot be presented in the form  $\mu=t\mu_1+(1-t)\mu_2$ for $t\in (0,1)$ and $\mu_1\ne \mu_2$ belonging to the same simplex. Let $\CRS^e(G) \subset \CRS(G)$ denote the subset of extreme points. In other words, a measure $\mu \in \CRS(G)$ is contained in $\CRS^e(G)$ if there does not exist measures $\mu_1\ne \mu_2 \in \CRS(G)$ and $t\in (0,1)$ such that $\mu=t\mu_1+(1-t)\mu_2$.

   They can be interpreted as ergodic measures in the sense of ergodic theory of group actions, but we have to make a warning at this point that the group $\Aut(G)$ (with the compact-open topology) may turn out to be not locally compact second countable or not countable discrete (and this will be the case in the situation related to Theorems \ref{thm:hcA} and \ref{thm:cAn}) and for non-locally compact groups the relationship of ergodicity and indecomposability is more complicated as for the first time was observed by Kolmogorov \cite{Fo50}. We will denote by $\IRS^e(G)$, $\CRS^e(G)$ etc. the sets of indecomposable measures. 

We recall the definition of ergodic action and of weakly mixing action for a general topological group in Section \ref{2.5}.

Once again, we call a measure $\mu \in RS$, IRS, CRS  etc. continuous if it has no atoms. One of the basic questions in this area is: given an interesting group $G$ and subgroups $\Psi\leq\Phi$ in $\Aut(G)$ does $G$ have any $\Psi$-ergodic  ($\Psi$-weakly mixing, $\Psi$-mixing etc.) $\Phi$-invariant random subgroups. 

For instance Theorem \ref{thm:free} deals with the case when $G=\F_r$ is a free group, $\Phi=\Aut(\F_r)$ is the group of all automorphisms of $\F_r$, and $\Psi=\Inn(\F_r)$ is the group of all inner automorphisms of $\F_r$ which we identify in a natural way with $\F_r$. It is clear that $\Psi$-ergodic and $\Psi$-weakly mixing implies $\Phi$-ergodic and $\Phi$-weakly mixing respectively.

Our first result in this paper is:
\begin{thm}\label{thm:hyp}
If $G$ is either
\begin{itemize}
\item a non-elementary Gromov hyperbolic group,
\item the mapping class group of a (possibly punctured by finitely many points) oriented surface of negative Euler characteristic, or
\item the outer automorphism group of a nonabelian free group
\end{itemize}
then $G$ has $2^{\aleph_0}$ many continuous $G$-weakly mixing IRS's.
\end{thm}

%The proof uses {\em characteristic random subgroups} (CRS): these are Borel probability measures on $\Sub(G)$ that are invariant under the full automorphism group, denoted $\Aut(G)$, of $G$ ($\Aut(G)$ acts on $\Sub(G)$ in the obvious way). We let $\IRS(G)$ ($\CRS(G)$) denote the space of $\Inn(G)$-invariant ($\Aut(G)$-invariant) Borel probability measures on $\Sub(G)$ with the weak* topology.  It is a natural question to ask which groups have continuous ergodic CRS's and what additional information can be collected about the simplices $\IRS(G)$ and $\CRS(G)$. To be precise, we will say that a measure on $\Sub(G)$ is {\em $G$-ergodic} ($G$-weakly mixing) if it is ergodic (weakly mixing) with respect to the $G$-action by inner automorphisms. With this terminology, we will avoid confusing $G$-ergodic with $\Aut(G)$-ergodic (for example). Our second result is: %It is called {\em $\Aut(G)$-indecomposable} if it is indecomposable with respect to the $\Aut(G)$-action. Clearly, a $G$-ergodic CRS is also $\Aut(G)$-ergodic but the converse need not hold. We define {\em $G$-weakly mixing} and {\em $\Aut(G)$-weakly mixing} similarly. 
This result we deduce from the next theorem.

\begin{thm}\label{thm:free}
Every non-abelian free group $\F_r$ (of finite or countably infinite rank) has $2^{\aleph_0}$ many continuous $\Inn(\F_r)$-weakly mixing CRS's.
\end{thm}

Let us see how Theorem \ref{thm:free} implies Theorem \ref{thm:hyp}:
\begin{proof}[Proof of Theorem \ref{thm:hyp} from Theorem \ref{thm:free}]
Let $G$ be a countable group and suppose $N\vartriangleleft G$ is a normal non-abelian free subgroup. By Theorem \ref{thm:free}, $N$ has $2^{\aleph_0}$ many continuous $N$-weakly mixing CRS's. However, each of these CRS's is also a $G$-weakly mixing IRS of $G$. This is because $N$ is normal in $G$, so $\Inn(G)$ naturally maps into $\Aut(N)$ and the image of $\Inn(G)$ in $\Aut(N)$ contains $\Inn(N)$ (also see Lemma \ref{properties}(a) below). So it suffices to show that if $G$ is as in Theorem \ref{thm:hyp} then $G$ has a free non-abelian normal subgroup. If $G$ is a non-elementary hyperbolic group, then this result is due to Delzant \cite{De96}. Otherwise, it is due to Dahmani-Guirardel-Osin \cite[Theorems 8.1, 8.5]{DGO}. 
\end{proof}

\begin{remark}
%Recall that a subgroup $H\le G$ is {\em characteristic} if it is invariant under all automorphisms of $G$.
 It follows from the proof above that if a group $G$ has trivial center and possesses a characteristic nonabelian free subgroup $N$ then $G$ has $2^{\aleph_0}$ many continuous $G$-weakly mixing CRS's. %Moreover, if $N$ is characteristic (i.e. invariant under $\Aut(G)$) then $G$ has $2^{\aleph_0}$ many continuous $G$-weakly mixing CRS's. 
\end{remark}

\begin{remark}
As  was  indicated  by  D.Osin the theorem  in fact holds   for  a much  larger  class  of  groups,  namely acylindrically  hyperbolic groups,  as they each contain a normal noncommutative free  subgroup. For  details  see      arXiv:1304.1246  and  arXiv:1111.7048.
\end{remark}

%\begin{remark}
%The proofs of Theorems \ref{thm:free} and \ref{thm:hyp} are constructive.
%\end{remark}

%The proof of Theorem \ref{thm:free} uses CRS's of abelian groups. Although we do not need to classify such CRS's in order to prove Theorem \ref{thm:free}, we have decided t

In order to prove Theorem \ref{thm:free}, we study CRS's of abelian groups. Although the proof of Theorem \ref{thm:free} uses only the existence of a continuous ergodic CRS's in the free elementary $p$-group ($p$-prime) $\oplus_\N (\Z/p\Z)$ of infinite rank, we think it is interesting to obtain a complete classification of indecomposable CRS's of the abelian groups $\A=\oplus_\N \Z,\, \A_n=\oplus_\N \Z/n\Z, \,\hcA=\prod_\N \R/\Z$ and $\hcA_n=\prod_\N \Z/n\Z$ (Theorems \ref{thm:hcA} and \ref{thm:cAn} below). %In fact our result has connections with the result of $\cite{GO09}$ on invariant random subspaces of infinite dimensional vector space over a finite field.

To explain, we need more notation. During this paper $m,n,i,k$ will denote nonnegative integers, $\Z_+=\{0,1,\dots\}$ and $\N=\{1,2,\dots\}$. For $m,n\in\Z_+$, we will write $m|n$ if $m,n\in \N$ and $m$ divides $n$. Note $m|0$ for any $m\in\Z_+$, but $0\not|n$ for $n\in\N$.

Given locally compact abelian groups $G,H$ let $\Hom(G,H)$ denote the set of continuous homomorphisms from $G$ to $H$. This is an abelian group under pointwise addition. We consider $\Hom(G,H)$ with the compact-open topology (\cite[23.34]{HR79}). Let $\T=\R/\Z$ be the one-dimensional torus and $\hat{G} = \Hom(G, \T)$ denote the {\em Pontryagin dual} of $G$. 

Given a subgroup $H\le G$, let $\Ann(H) = \{ \alpha \in \hat{G}:~ \alpha(h) = 0~ \forall h\in H\}$ be its {\em annihilator} subgroup. Given a subgroup $S\le \hat{G}$, let $\Ker(S)=\{g\in G:~s(g)=0~\forall s \in S\}$ denote its {\em kernel}.

Let $\cA=\oplus_\N \Z$ be the free abelian group of countable rank. Its Pontryagin dual group $\hcA$ is naturally isomorphic with $\prod_\N \T=\T^\N$, the infinite-dimensional torus: if $x=\{x_i\} \in \cA$, $y =\{y_i\}\in \T^\N$ then $y(x) := \sum_{i\in \N} x_iy_i \in \T$. For any $n\in\Z_+$, let $n\cA  \le  \cA$ be the subgroup of elements with all coordinates divisible by $n$. Let $\cA_n=\cA/n\cA \cong \oplus_\N \Z/n\Z$. Then for $n\in\N$, the annihilator subgroup $\Ann(n\cA)\le \hcA$ is naturally isomorphic with $\prod_\N \Z[1/n]/\Z=(\Z[1/n]/\Z)^\N$ which is naturally isomorphic with the dual $\hcA_n$ (Lemma \ref{lem:Ann}). Observe that the subgroups $n\cA\le \cA$  are {\em characteristic}: they are invariant under all automorphisms of $\cA$. Similarly, $\Ann(n\cA)=\hcA_n\le \hcA$ is characteristic. Indeed, these are all the characteristic subgroups (Lemma \ref{lem:characteristic}).

In the future we will identify $\hcA_n$ with $\Ann(n\A)\le \hcA$. If $m,n\in\N$ and $m|n$ then $\hcA_m=\Ann(m\A_n)=(\Z[1/m]/\Z)^\N\leq (\Z[1/n]/\Z)^\N=\hcA_n$. In the case $n=0$, $\Ann(0\A)=\hcA$. Thus $\hcA_0=\hcA$. Also $\A_0=\A/0\A=\A$ and $\A_1=\hcA_1$ is a trivial group.

% In the next two results, we classify all extreme CRS's of $\hcA_n,\cA_n$, $n\in\Z_+$. Observe that since $\hcA_n$ for all  $n\in\N$ is a closed characteristic subgroup of $\hcA$, it follows that $\CRS^e(\hcA_n)\subset \CRS^e(\hcA)$. Therefore the description of elements of $\CRS^e(\hcA)$ will give a description of elements of $\CRS^e(\hcA_n)$, as long as we will be able to characterize which elements of $\CRS^e(\hcA)$ in fact lie in $\CRS^e(\hcA_n)$. We will state our results in the form that covers all cases but the proof will be concentrated on the case $n=0$, i.e. $G=\A$ or $\hcA$.

 %First we need a definition and some notation.
 \begin{defn}\label{defn:over}
Let $m\in \Z_+$ and $F$ a finite abelian group. We say that {\it $F$ is over $m$} if either $F=0$ or each nontrivial summand $F_1$ of $F$ satisfies $mF_1\neq 0$. Note that in particular, any $F$ is over $1$ and only the trivial group is over $0$.
\end{defn}
Note that if $F$ is a finite abelian group, then by the main structure theorem for finite abelian groups, $F$ is isomorphic to $\oplus_{i=1}^s\Z/p_i^{r_i}\Z$, where $p_i$ are primes. In this case $F$ is over $m$ if and only if for all $i$ we have that $p_i^{r_i}$ does not divide $m$.

Note that when $F=\oplus_{i=1}^s\Z/p_i^{r_i}\Z$ and $nF=0$ it follows that $p_i^{r_i}|n$ and thus we have 
\begin{equation*}
\begin{aligned}
&\Hom(F,\hcA_n)\simeq\oplus_{i=1}^s\Hom(\Z/p_i^{r_i}\Z,\hcA_n)\simeq\oplus_{i=1}^s\Hom(\Z/p_i^{r_i}\Z,\Z[1/n]/\Z)^\N\simeq\\
&\oplus_{i=1}^s(\Z/p_i^{r_i}\Z)^\N\simeq\oplus_{i=1}^s\hcA_{p_i^{r_i}},
\end{aligned}
\end{equation*}
\begin{equation*}
\begin{aligned}
&\Hom(\A_n,F)\simeq\oplus_{i=1}^s\Hom(\A_n,\Z/p_i^{r_i}\Z)\simeq\oplus_{i=1}^s \Hom(\Z[1/n]/\Z,\Z/p_i^{r_i}\Z)^\N\simeq\\
&\oplus_{i=1}^s(\Z/p_i^{r_i}\Z)^\N\simeq\oplus_{i=1}^s\hcA_{p_i^{r_i}},
\end{aligned}
\end{equation*}
where the isomorphisms are in fact homeomorphisms (\cite[23.34]{HR79}). The Haar measure on $\Hom(F,\hcA_n)$ and $\Hom(\A_n,F)$ equals the product of Haar measures on $\hcA_{p_i^{r_i}}$. %The Haar measure on $\hcA_n=(\Z[1/n]/\Z)^\N$, $n\in\N$, is the product measure of the uniform measures on $\Z[1/n]/\Z\simeq \Z/n\Z$.

 Given a group $H$, let $[H]$ denote its isomorphism class.
\begin{thm}\label{thm:hcA}
 Let $\lambda \in \CRS^e(\hcA_n)$. Then there is a unique pair $(m,[F])$ with $m\in \Z_+$, $F$ a finite abelian group over $m$, such that $m|n$, $nF=0$ and the random subgroup $\Ann(m\A_n)+h(F)$ has law $\lambda$, where $h\in\Hom(F,\hcA_n)$ is a random homomorphism  with law equal to the Haar measure on the compact group $\Hom(F,\hcA_n)$. 
 \end{thm}

In the next Theorem we give a characterisation of $\CRS^e(\A_n)$.
\begin{thm}\label{thm:cAn}
 Let $\lambda \in \CRS^e(\A_n)$. Then there is a unique pair $(m,[F])$ with $m\geq 0$, $F$ a finite abelian group over $m$, such that $m|n$, $nF=0$ and the random subgroup $m\A_n\cap \Ker(h)$ has law $\lambda$, where $h\in\Hom(\cA_n,F)$ is a random homomorphism with law equal to the Haar measure on the compact group $\Hom(\cA_n,F)$. 
  \end{thm}
 Observe that in both Theorems \ref{thm:hcA} and \ref{thm:cAn} the set $\CRS^e(G)$ is infinite countable. In the case of groups $G=\hcA$ or $G=\A$ the sets $\CRS^e(G)$ can be parametrized by  $m\in \Z_+$ and a tuple of the form $(p_1^{r_1},\dots,p_k^{r_k})$ were $p_i$ are primes, $k_i\in \N$, $p_i\leq p_{i+1}$ and $r_i\leq r_{i+1}$ if $p_i=p_{i+1}$ (such tuples parametrize finite abelian groups).
For the case of groups $G=\hcA_n$ or $G=\A_n$, $n\geq 2$, the sets $\CRS^e(G)$ can be parametrized with the same tuples with the additional restriction that both $m$ and all $p_i^{r_i}$ must divide $n$.

  We thus have that for all cases the set $\CRS^e(G)$ is countable. Moreover the set $\CRS^e(G)$ is closed in $\CRS(G)$ (see Corollary \ref{cor:top}), and hence $\CRS(G)$ is a Bauer simplex in all considered cases (recall that a simplex is called Bauer if the set of extreme points is closed).

Let $n=p\ge 2$ be prime. In this case Theorem \ref{thm:cAn} takes on a particularly simple form. Indeed, if $m|p$ then $m=1$ or $m=p$. If $m=p$, then $F$ over $p$ and $pF=0$ implies that $F=0$. This corresponds to the CRS concentrated on $\{0\}\leq \A_p$. If $m=1$, then any $F$ is over $m$. Since $pF=0$ we have that $F = (\Z/p\Z)^k$ for some $k \ge 1$. Thus every nontrivial ergodic CRS of $\cA_p$ has the form $\Ker(h)$ where $h\in\Hom(\A_p,(\Z/p\Z)^k)$. This implies the random subgroup of $\cA_p$ has index $p^k$ almost surely. Moreover, the CRS is completely determined by the exponent $k$. This special case was obtained earlier by Gnedin-Olshanski \cite{GO09} by completely different methods. Indeed, the idea of \cite{GO09} is to interpret $\cA_p$ as a vector space over the field of order $p$ and obtain a recursive formula for the finite-dimensional marginals. Our method is based on the de Finetti-Hewitt-Savage Theorem (e.g. see \cite{G}) and duality and avoids explicit computations. 

In \cite{GO09} A. Gnedin and G. Olshanski provided a characterisation of random spaces over the Galois field $\F_q$ that are invariant under the natural action of the infinite group of invertible matrices with coefficients from $\F_q$. Our methods allow us to give a different, probabilistic proof of their results. Let us formulate the results.

Let $q$ be a power of a prime and $\F_q$ be the corresponding finite field with $q$ elements. Let $V$ be a locally compact vector space over $\F_q$. Denote by  $\Subp(V)$ the space of closed subspaces of $V$. Let $\Aut_{\F_q}(V)$ be the group of all homeomorphic invertible linear maps $V\to V$. Denote by $\CRS_{\F_q}(G)$ the Borel probability measures on $\Subp(V)$ invariant under $\Aut_{\F_q}(V)$ (these are exactly the laws of random spaces of \cite{GO09}). We have
\begin{thm}\label{thm:hcB}
 Let $\lambda \in \CRS_{\F_q}^e( (\F_q)^\N)$, $\lambda\neq \delta_{(\F_q)^\N}$. Then there is a unique $\kappa\in \Z_+$ such that the random subgroup $h(\F_q^\kappa)$ has law $\lambda$, where a random homomorphism $h\in\Hom_{\F_q}(\F_q^\kappa,(\F_q)^\N)$ is given by the Haar measure on the compact group $\Hom_{\F_q}(\F_q^\kappa,(\F_q)^\N)$. 
 \end{thm}

%In the next Theorem we give a characterisation of $\CRS_{\F_q}^e(\oplus_\N \F_q)$.
\begin{thm}\label{thm:cBn}
 Let $\lambda \in \CRS_{\F_q}^e(\oplus_\N \F_q)$, $\lambda\neq \delta_{\oplus_\N \F_q}$. Then there is a unique $\kappa\in\Z_+$ such that the random subgroup $\Ker(h)$ has law $\lambda$, where a random homomorphism $h\in\Hom_{\F_q}(\oplus_\N \F_q,\F_q^\kappa)$ is given by the Haar measure on the compact group $\Hom_{\F_q}(\oplus_\N \F_q,\F_q^\kappa)$. 
  \end{thm}
  Denote by $\mu_\kappa\in\CRS_{\F_q}^e(\oplus_\N \F_q)$ the measure in Theorem \ref{thm:cBn}. Define $\tilde{v}_{n,k}=\mu_\kappa(\{X\,|\, \dim_{\F_q}(X\cap V_n)=k\})$. 
We have 
\begin{thm}\label{appro}
Suppose that $n\geq k\geq 0,$ and $\kappa\geq n-k$. Then $\tilde{v}_{n,k}$ is equal to the number of $\F_q$-matrices of size $\kappa\times n$ which have rank $n-k$ divided by the number of all $\F_q$-matrices of size $\kappa\times n$. For all other pairs $(n,k)$, $\tilde{v}_{n,k}=0$.
\end{thm}
We also give an explicit formula for $\tilde{v}_{n,k}$ in Corollary \ref{computation}.

{\bf Organization}
In \S \ref{sec:prelim}, we go over background material. A first-time reader may choose to skip this section and refer back to it later as need be. In \S \ref{sec:free} we prove Theorem \ref{thm:free}. In \S \ref{sec:homom} - \S \ref{sec:abelian} we prove Theorems \ref{thm:hcA}-\ref{thm:cAn}. %The appendix contains an alternative proof of Theorems \ref{thm:hcA}-\ref{thm:cAn}.  
In \S \ref{GO1} and \S\ref{GO2} we sketch the proofs of Thorems \ref{thm:hcB}-\ref{thm:cBn} and compute the numbers $\tilde{v}_{n,k}$.

{\bf Acknowledgements}
The first two authors would like to thank the organizers of the workshop ``Group Theory, Measure, and Asymptotic Invariants'' at Mathematisches Forschungsinstitut Oberwolfach where this research began. The first author would especially like to thank Miklos Abert for explaining a construction of CRS's in $\cA_p$ where $p$ is prime. The second author would like to thank Mark Sapir for discussions related to fully invariant subgroups and \cite{A}. The second and third authors acknowledge on the support of Institute of Henri Poincare at Paris, as the final work on this note has been done during trimester program  "Random Walks and Asymptotic Geometry of Groups". The second author also acknowledges support of University Paris 6 during the same visit.

\section{Preliminaries}\label{sec:prelim}
As the scope of this paper is on the border between group theory, probability theory and dynamical systems, we recall some basic definitions related to these fields.
 
%  \subsection{The space of subgroups and $\CRS(G)$}
 %For a locally compact second countable group $G$, let $\Sub(G)$ denote the set of all closed subgroups on $G$ with the Chabauty topology. To be precise, a sequence $\{H_i\}_{i=1}^\infty \subset \Sub(G)$ converges to $H_\infty \in \Sub(G)$ if and only if for every compact $K \subset G$, $\lim_{i\to\infty} H_i \cap K = H_\infty \cap K$ in the topology given be the Hausdorff metric on the space of closed subsets of $K$. Here the Hausdorff metric is with respect to any continuous left-invariant metric on $G$. Because $G$ is locally compact and second countable, $\Sub(G)$ is a compact metrizable space. 
 
\subsection{Compact-Open topology and Braconnier topology}
Let $X$ and $Y$ be topological spaces. Denote by $C(X,Y)$ the set of all continuous maps from $X$ to $Y$.
\begin{defn}
{\it The compact-open topology} on $C(X,Y)$ is generated by sets $V(K,U)=\{f\in C(X,Y)\,|\, f(K)\subseteq U\}$, where $K\subseteq X$ is a compact set and $U\subseteq Y$ is an open set.
\end{defn}
Considering this topology on $C(X,Y)$, the following property is well-known
\begin{lem}\label{cont}
Suppose $X$ is locally compact Hausdorff. Then the evaluation map $C(X,Y)\times X\to Y$, given by $(f,x)\mapsto f(x)$, is continuous.
\end{lem} 

Suppose now that $G$ is a locally compact second countable group. Denote by $\Aut(G)$ the set of all homeomorphic automorphisms of $G$.  Since $\Aut(G)\subset C(G,G)$, we can endow $\Aut(G)$ with topology by restricting the compact-open topology on $C(G,G)$. However, in this topology, $\Aut(G)$ is not necessarily a topological group.

\begin{defn}
{\it The Braconnier topology} on $\Aut(G)$ is the smallest refinement of the compact-open topology, such that the inverse map $\phi\mapsto \phi^{-1}$ is continuous.
\end{defn}
%We have the obvious lemma
%\begin{lem}\label{brac}
%$\phi_i\to \phi$ in the Braconnier topology if and only if both $\phi_i\to\phi$ and $\phi_i^{-1}\to\phi^{-1}$ in the compact open topology.
%\end{lem}

It is known that for compact or locally connected (in particular discrete) $G$, the Braconnier topology coincides with the compact-open topology. That is, the inverse map on $\Aut(G)$ is automatically continuous in the compact-open topology.
 
 The Braconnier topology turns $\Aut(G)$ into topological group, and from now on we will always consider $\Aut(G)$ with the Braconnier topology.
 
 \begin{lem}\label{eval}
The evaluation map $e:\Aut(G)\times G\to G$ is continuous.
 \end{lem}
 \begin{proof}
 Note that if we consider the compact-open topology on $\Aut(G)$, then $e$ is continuous by Lemma \ref{cont} and the assumption that $G$ is locally compact topological group (hence Hausdorff). It is obvious then that $e$ is continuous for any refinement of the compact-open topology. 
  \end{proof} 
 The next proposition shows that the action of $\Aut(G)$ on $\Sub(G)$ is continuous. First we need a lemma.
 \begin{lem}[Proposition E.1.2 in \cite{BP92}]\label{top_sub}
 Let $G$ be locally compact second countable group. A sequence $\{H_n\} \subset \Sub(G)$ converges to $H \in \Sub(G)$  if and only if two conditions hold: 
 \begin{itemize}
\item[a)] if $g\in G$ and there exist $g_i\in H_{n_i}$ such that $g_i\to g$, then $g\in H$,
\item[b)] if $g\in H$ then there exist $g_n\in H_n$ such that $g_n\to g$.
\end{itemize}
\end{lem}
 
 \begin{prop}\label{cont_eval}
 The map $\Aut(G)\times\Sub(G)\to \Sub(G)$, given by $(\phi,H)\mapsto \phi(H)$, is continuous.
 \end{prop}
 \begin{proof}

Suppose that $\phi_n\to \phi$ in the Braconnier topology and $H_n\to H$ in $\Sub(G)$.  We need to show that $\phi_n(H_n)\to \phi(H)$ in $\Sub(G)$.

Note that we have also that $\phi_n^{-1}\to \phi^{-1}$ in the Braconnier topology.

Assume that $g_i=\phi_{n_i}(h_{n_i})\in \phi_{n_i}(H_{n_i})$ and $g_i\to g$. Then $h_{n_i}=\phi^{-1}_{n_i}(g_i)\to \phi^{-1}(g)$, by Lemma \ref{eval}. It follows that $\phi^{-1}(g)\in H$ and hence $g\in \phi(H)$.

Assume now that $g\in \phi(H)$. Then $g=\phi(h)$ for some $h\in H$. It follows that there exist $h_n\in H_n$ such that $h_n\to h$. We then have that $\phi_n(h_n)\to \phi(h)=g$, by Lemma \ref{eval}.

 \end{proof}

 \subsection{Abelian groups and Pontryagin duality}\label{sec:duality}

We will express the group operations in abelian groups additively. For example, if $G$ is a locally compact second countable abelian group, $n \in \N$ and $x\in G$ then $nx = x+ x +\cdots +x$ ($n$ times). Let $\T = \R/\Z$ and $\hat{G} = \Hom(G,\T)$ denote the {\em Pontryagin dual} of $G$. Note that $\Hom(G,\T)$ is itself a group under pointwise addition and it is locally compact under the compact-open topology. If $G$ is discrete then $\hat{G}$ is compact and vice versa. Moreover, $\hat{\hat{G}}=G$.
% Let $\Aut(G)$ denote the group of continuous group automorphisms of $G$ with the compact-open topology. If $G$ is discrete and countable then this topology is the same as the pointwise convergence topology.

%We note the following exercise.
\begin{prop}\label{prop:Pontryagin}
The groups $\Aut(G)$ and $\Aut(\hat{G})$ are naturally isomorphic as topological groups. Indeed the map $\psi \mapsto \hat{\psi}^{-1}$ is a continuous isomorphism (with continuous inverse) where
$$\hat{\psi}(h)(x)=h(\psi(x)).$$
\end{prop}
\begin{proof}
Note that $\psi_i\to\psi$ in the Braconnier topology if and only if both $\psi_i\to\psi$ and $\psi_i^{-1}\to\psi^{-1}$ in the compact-open topology. 

To put this question in a more general context, let $G,H$ be any locally compact abelian groups. For $\psi \in \Hom(G,H)$ define $\hat{\psi} \in \Hom(\hat{H},\hat{G})$ by
$$\hat{\psi}(\phi)(g) = \phi(\psi(g))\quad \phi\in \hat{H}, g\in G.$$
It now suffices to show that if $\phi_i\to\phi$ in $\Hom(G,H)$ then $\hat{\phi}_i\to \hat{\phi}$ in $\Hom(\hat{H},\hat{G})$ (in the compact-open topology).  Since the map $\psi \mapsto \hat{\psi}$ is linear, this is equivalent to: if $\phi_i\to 0$ then $\hat{\phi}_i\to 0$. Thus the proof is finished as soon as the next lemma is proved.
\begin{lem}\label{cont_1}
$\phi_i\to 0$ implies that $\hat{\phi}_i\to 0$.
\end{lem}
\end{proof}

To prove Lemma \ref{cont_1} we need a sequence of lemmas. 

Let $U=\{z\in\T\,:\,|z|<1/8\}$ be a small neighborhood of $0\in \T$.
\begin{lem}\label{L1}
Let $K\subseteq G$ be a compact subset. Define $K^*=\{h\in \hat{G}\,|\, h(K)\subseteq\bar{U}\}$. Then $K^*$ is a closed compact neighborhood of $0\in\hat{G}$.
\end{lem}
\begin{proof}
 The map $G\times\hat{G}\to \T$, given by $(g,h)\mapsto h(g)$ is continuous by Lemma \ref{cont}. We have that $0(K)=0\in U$, and thus $0\in K^*$ and by continuity a neighborhood of $0$ is in $K^*$. Again by continuity $K^*$ is closed.
 
 Suppose that there are $h_n\in K^*$ such that $h_n\to\infty$. We have the inclusion $\hat{G}\subset L^\infty(G)$, and the topology on $\hat{G}$ is the restriction of the weak* topology on $L^\infty(G)=(L^1(G))^*$. Moreover, the closure of $\hat{G}$ in $L^\infty(G)$ is the compact set $\hat{G}\cup \{0\}$ (see \cite[proof of Theorem 4.2]{F95}).  Thus for any $f\in L^1(G)$ we have $\int f(g) \exp(2\pi i h_n(g))\,dg\to 0$. Take $f$ to be the characteristic function of $K$. Then $\int f(g) \exp(2\pi i h_n(g))\,dg=\int_K \exp(2\pi i h_n(g))\,dg $, and so 
 \begin{eqnarray*}
 \left| \int_K \exp(2\pi i h_n(g))\,dg -\textrm{Haar}(K)\right|&\leq& \int_K \left| \exp(2\pi i h_n(g))-1 \right| \,dg \\
 & \leq& \int_K |\exp(\pi i/4) - 1|\,dg \leq \frac{\pi}{4}\textrm{Haar}(K),
\end{eqnarray*}
 a contradiction.
\end{proof}
\begin{lem}
Let $X$ be a topological space. Suppose that $K_n$ are compact, $K_n\supseteq K_{n+1}$ and $\cap_n K_n=\{x\}$. Let $U$ be an open neighborhood of $x$. Then there is $n$ such that $K_n\subset U$.
\end{lem}
\begin{proof}
Suppose that for all $n$ there is $x_n\in K_n-U$. Since $K_n\subseteq K_1$, which is compact, we can choose a subsequence $x_{n_i}\to x_0$. Since $U$ is open and  $x_n\not\in U$, we have that $x_0\not\in U$. On the other hand, for every $m$ we have that $x_0\in K_m$. Indeed, since $K_{n+1}\subseteq K_{n}$, we have that for all $m$ there is an $I$ such that $i>I$ implies $x_{n_i} \in K_m$. Thus $x_0\in \cap_m K_m=\{x\}$, a contradiction.
\end{proof}
%\begin{lem}\label{compact}
%Suppose $K_n\subseteq G,$ $n\in \Z$ are compact subsets, such that $K_n\subseteq K_{n+1}$ for $n\in\Z$, $\cup_n K_n=G$ and $\cap_n K_n=\{0\}$. Then $K^*_n\supseteq K^*_{n+1}$ for $n\in\Z$, $\cap_n K^*_n=\{0\}$ and $\cup_n K^*_n=\hat{G}$.
%\end{lem}
\begin{lem}\label{compact}
Suppose $K_n\subseteq G,$ $n\in \N$ are compact subsets, such that $K_n\subseteq K_{n+1}$ for $n\in\N$ and $\cup_n K_n=G$. Then $K^*_n\supseteq K^*_{n+1}$ for $n\in\N$ and $\cap_n K^*_n=\{0\}$.
\end{lem}
\begin{proof}
Let $h\in \cap_n K^*_n$. Then $h(G)\subseteq \bar{U}$, and thus, since $\bar{U}$ contains no subgroup of $\T$ beside $\{0\}$, we have that $h(G)=0$, and thus $h=0$.

%Now let $h\in\hat{G}$. We have that $0=h(0)=h(\cap_n K_n)=\cap_n h(K_n)$, since $K_n$ are compact. Since $h(K_n)$ are compact subsets of $\T$ intersecting to $0$, there is $n$ such that $h(K_n)\subset U,$ and so $h\in K^*_n$.
\end{proof}
\begin{proof}[Proof of Lemma \ref{cont_1}]
Suppose that $\phi_i:G\to H$ converges to $0\in \Hom(G,H)$ in the compact-open topology. We need to show that $\hat{\phi}_i: \hat{H}\to \hat{G}$ converges to $0$. Let $C \subset \hat{H}$ be compact and $V \subset \hat{G}$ be an open neighborhood of $0$. It suffices to show that $\hat{\phi}_i(C) \subset V$ for all sufficiently large $i$.

%Equivalently, we need to show that for a compact subset $C\subset \hat{H}$ and an open neighborhood of $0$ $V\subset \hat{G}$, there is $I$ such that $i\geq I$ implies $\hat{\phi}_i(C)\subset V$. 

%Let $V'\subset \hat{G}$ be an open neighborhood of $0$ whose closure is contained in $V$.
 Because $G$ is locally compact second countable, $G$ is $\sigma$-compact. So there exists an increasing sequence of compact subsets $K_n\subset G$ satisfying $\cup_n K_n=G$. By Lemma \ref{compact}, $\cap_n K_n^*=\{0\}$. Since $K^*_n$ are compact, there is an $n_0$ such that $K^*_{n_0}\subset V$. 

In order to show that $\hat{\phi}_i(C)\subset V$ it suffices to check that for each $x\in K_{n_0}$ and $y \in C$,  $\hat{\phi}_i(y)(x)\in \bar{U}$. Indeed, this implies $\hat{\phi}_i(C)\subset K^*_{n_0}\subset V$. 

We have that $\hat{\phi}_i(y)(x)=y(\phi_i(x))$. Therefore it suffices to check that $\phi_i(K_{n_0})\subseteq C^*$. By Lemma \ref{L1}, $C^*$ is a neighborhood of $0$, and thus, since $\phi_i\to 0$, there exists $I$ such that $i\geq I$ implies that $\phi_i(K_{n_0})\subseteq C^*$. 
\end{proof}
Because of Proposition \ref{prop:Pontryagin}, we will not distinguish between $\Aut(G)$ and $\Aut(\hat{G})$.

%\begin{proof}
%This is an exercise.
%\end{proof}

Define $\Ann:\Sub(G) \to \Sub(\hat{G})$ and $\Ker:\Sub(\hat{G}) \to \Sub(G)$ by
$$\Ann(X) = \{ h \in \hat{G}:~ h(x)=0~ \forall x\in X\}$$
$$\Ker(H) = \{x\in G:~h(x)=0~\forall h \in H\}.$$
It is easy to check using Lemma \ref{top_sub} that both $\Ann$ and $\Ker$ are continuous maps. Thus they are continuous inverses of each other and they are $\Aut(G)$-equivariant. So they induce homeomorphisms $\Ann:\Char(G) \to \Char(\hat{G})$, $\Ker:\Char(\hat{G}) \to \Char(G)$ and affine homeomorphisms $\Ann_*: \CRS(G) \to \CRS(\hat{G})$ and $\Ker_*:\CRS(\hat{G}) \to \CRS(G)$. Also $\Ann$ and $\Ker$ are order-reversing in the sense that $H\le K$ implies $\Ann(H) \ge  \Ann(K)$ and similarly with $\Ker$. Lastly, they take addition to intersection and vice versa. To be precise, suppose $H,K$ are subgroups of  $\hat{G}$ and $H',K'$ are subgroups of $G$. Then
\begin{eqnarray}\label{eqn:dual}
\Ker(H+K) = \Ker(H) \cap \Ker(K) &&\quad \Ker(H \cap K) = \Ker(H)+\Ker(K)\\
\Ann(H'+K')= \Ann(H') \cap \Ann(K') &&\quad \Ann(H' \cap K') = \Ann(H')+\Ann(K') \nonumber.
\end{eqnarray}

%A similar statement holds with $\Ann$ in place of $\Ker$. 
%Next we obtain an application of the formulas above. First we need some definitions.

%The equations above imply:
%\begin{lem}\label{lem:duality1}
%For any measure $\lambda \in \CRS(\hcA)$, $\Ker(H_\lambda) = N_{\Ker_*\lambda}$. % where $H_\lambda, N_{\Ker_*\lambda}$ are as in Definition \ref{defn:hull}. 
%\end{lem}

\begin{lem}\label{lem:Ann}
For any $n\ge 2$, $\Ann(n\cA) = (\Z[1/n]/\Z)^\N \le \T^\N$.
\end{lem}

\begin{proof}
Recall from the introduction that if $x=\{x_i\} \in \cA = \oplus_\N \Z$ and $y =\{y_i\} \in \hcA  = \T^\N$ then $y(x):=\sum_{i\in \N} x_iy_i \in \T$. We have that $x\in n\cA$ if and only there exists $x' \in \cA$ with $x = nx'$. If $x'= \{x'_i\}$ and $y \in (\Z[1/n]/\Z)^\N$ then $y(x) = \sum_{i \in \N} nx'_iy_i = \Z$ because $ny_i = \Z$ for all $i$. This implies $\Ann(n\cA) \ge  (\Z[1/n]/\Z)^\N$. To see the other direction, suppose $y \in \Ann(n\cA)$. Fix $j \in \N$ and let $x=\{x_i\}$ be such that $x_i=n$ if $i=j$ and $x_i=0$ otherwise. Then $x \in n\cA$ so $y(x)= 0 = x_jy_j = ny_j$. Therefore, $y_j \in \Z[1/n]/\Z$. Since $j$ is arbitrary, $y \in (\Z[1/n]/\Z)^\N$. Since $y$ is arbitrary, $\Ann(n\cA) \le(\Z[1/n]/\Z)^\N$. 
\end{proof}

%\begin{lem}\label{lem:automorphisms}
%Because $\hcA_n$ is characteristic in $\hcA$ the inclusion map $i:\hcA_n \to \hcA$ induces a homomorphism $i^*:\Aut(\hcA) \to \Aut(\hcA_n)$. This homomorphism has dense image.
%\end{lem}

%\begin{proof}
%We identify $\hcA$ with the infinite-dimensional torus $\T^\N$ and $\hcA_n$ with $(\Z[1/n]/\Z)^\N$. For any $m\in \N$, $\Aut(\T^m)$ embeds into $\Aut(\hcA)$: it is the subgroup that fixes all $k$-th coordinates for $k>m$. The subgroup $(\Z[1/n]/\Z)^m$ is characteristic in $\T^m$. So the inclusion map $(\Z[1/n]/\Z)^m \to \T^m$ induces a homomorphism $\Aut(\T^m) \to \Aut((\Z[1/n]/\Z)^m)$. The groups $\Aut(\T^m)$ and  $GL(m,\Z)$ are canonically isomorphic as are the groups $\Aut((\Z[1/n]/\Z)^m)$ and $GL(m,\Z/n\Z)$. The following diagram commutes:

%which surjects onto $GL(m,\Z/n\Z)$ which is isomorphic to $\Aut( (\Z[1/n]/\Z)^m)$. In this way we see that $\Aut(\T^m)$ surjects onto $\Aut((\Z[1/n]/\Z)^m)$.

%\end{proof}

% $\Aut(\T^n)$ embeds into $\Aut(\hcA)$, that $\Aut(\T^n)$ surjects $\Aut( (\Z/n\Z)^m)$, that $\cup_n \Aut(\T^n)$ is dense in $\Aut(\hcA)$ and that $\Aut(\hcA)$ has dense image in $\Aut(\hcA_n)$. ***

 \subsection{Characteristic subgroups}
 
Let $\Char(G)$ denote the set of closed characteristic subgroups of $G$. 
\begin{lem}\label{lem:characteristic}
$\Char(\cA)=\{ r\cA:~ r =0,1,2,\ldots\}$, $\Char(\hcA)=\{ \Ann(r\cA):~r=0,1,2,\ldots\}$,  $\Char(\cA_n) = \{0\} \cup \{r\cA_n:~r \mid n, r \in \N\}$, $\Char(\hcA_n) = \{\hcA_n\} \cup \{\Ann(r\cA_n):~r\mid n, r\in \N\}$.
\end{lem}
\begin{proof}
Let us use the following terminology. An element $g\in \cA$ is {\em primitive} if there exists a subgroup $H\le\cA$ such that $\cA=C\oplus H$ where $C$ is the cyclic subgroup generated by $g$. An element $h\in \cA \setminus \{0\}$ has {\em co-order} $r \ge 1$ if there exists a primitive $g\in \cA$ such that $rg=h$. Equivalently, $h$ has co-order $r$ if, in the standard basis $h=(h_1,h_2,\ldots) \in \oplus_\N \Z$ and $r=\gcd(h_1,h_2,\ldots)$. We observe that $\Aut(\cA)$ acts transitively on primitive elements and therefore acts transitively on elements of co-order $r$ and the subgroup generated by all elements of co-order $r$ is $r\cA$. Thus $r\cA \in \Char(\cA)$ for $r=0,1,\ldots$.

Let $K \le \cA$ be characteristic. If $K$ contains an element of co-order $r$ then it must contain all such elements and so $r\cA \le K$. Since every nonzero element has a co-order, it follows that $K$ is a sum of subgroups of the form $r\cA$.  Since $r_1\cA + r_2\cA = \gcd(r_1,r_2)\cA$, this implies $\Char(\cA)=\{ r\cA:~ r =0,1,2,\ldots\}$. By duality, $\Char(\hcA)=\{ \Ann(r\cA):~r=0,1,2,\ldots\}$. 

Recall that $\cA_n = \cA/n\cA$. Let $\pi:\cA \to \cA_n$ be the quotient map. If $K\le\cA_n$ is characteristic then $\pi^{-1}(K)\le\cA$ is characteristic (because $n\cA$ is chacteristic in $\cA$). So $\pi^{-1}(K) = r\cA$ for some $r=0,1,2,\ldots$ which implies $K=\gcd(r,n)\cA_n$. This classifies characteristic subgroups of $\cA_n$. The result for $\hcA_n$ follows by duality.

\end{proof}

\begin{thm}\label{thm:Adyan}
Every non-abelian free group of finite or countable rank has $2^{\aleph_0}$ many characteristic (in fact, fully invariant) subgroups. Moreover, we can choose these subgroups to lie in the commutator subgroup.
\end{thm}

\begin{proof}
This result is a corollary of a famous theorem of Adyan \cite{A}. It can also be derived from \cite{Ol70} although we find it simpler to obtain the result from \cite{A}. We explain the necessary background next.

A subgroup $H\le G$ is {\em fully invariant} if $\varphi(H)\le H$ for every endomorphism $\varphi:G \to G$. This condition is stronger than being characteristic.

 An efficient method for contructing fully invariant subgroups is to use laws.  Let $X=\{x_1,x_2,\ldots\}$ be a countable alphabet whose elements are considered to be variables taking values in a countable group $G$ and let $\F(X)$ denote the free group with generating set $X$. For any word
$$w = x_{a_1}^{b_1}\cdots x_{a_n}^{b_n} \in \F(X)$$
let $f_w:G^n \to G$ be the corresponding {\em word map}:
$$f_w(g_1,\ldots,g_n) = g_{a_1}^{b_1}\cdots g_{a_n}^{b_n}.$$
If $W \subset \F(X)$ is any subset, we let $G(W)\le G$ be the subgroup generated by the union of the images $\cup_{w\in W} f_w(G)$. This subgroup is fully invariant and the quotient group $G/G(W)$ satisfies the laws $w=1$ $(w\in W)$. Every fully invariant subgroup of a free group can be obtained this way \cite[Theorem 12.34]{Neu67}. 

 A word $u \in \F(X)$ is a {\em consequence} of a set $W \subset \F(X)$ if $G(\{u\}) \subset G(W)$ for all $G$. In other words, if any group which satisfies the laws $w=1$ for $w\in W$ necessarily satisfies the law $u=1$. 

Let $\F_\infty$ denote the free group with countable rank. Suppose $u\in \F(X)$, $W \subset \F(X)$ and $\F_\infty(\{u\}) \subset \F_\infty(W)$. We claim that $u$ is a consequence of $W$. To see this, let $G$ be any countable group and let $\phi:\F_\infty \to G$ be a homomorphism with the property that for every nontrivial $g\in G$ there is a generator $h \in \F_\infty$ with $\phi(h)=g$. Then $\phi(\F_\infty(W))=G(W)$ and $\phi(F_\infty(\{u\}))=G(\{u\})$. So $G(\{u\}) \subset G(W)$ which implies $u$ is a consequence of $W$ as claimed. 

A subset $I \subset \F(X)$ is called {\em independent} if there does not exist $u\in I$ such that $u$ is a consequence of $I\setminus \{u\}$. We claim that if $I$ is independent then the subgroups $\{\F_\infty(A):~A \subset I\}$ are distinct and fully invariant (and therefore, characteristic) subgroups. To see this, suppose $A,B \subset I$ and $\F_\infty(A)=\F_\infty(B)$. It suffices to show that $A=B$. Suppose $b \in B \setminus A$. Then $\F_\infty(\{b\}) \le \F_\infty(B) = \F_\infty(A) \le \F_\infty(I \setminus \{b\})$. So $b$ is a consequence of $I \setminus \{b\}$ and therefore $I$ is not independent. This contradiction shows that $B \subset A$. By symmetry $A \subset B$ and so $A=B$ as required.

So for the countable rank case, it suffices to show that there exists an infinite independent subset of $\F(X)$. This problem had been open for a long time before Adyan produced a specific example in \cite{A}. In fact he shows that any odd number $n \ge 4381$, the set
$$I_n=\{ (x_1^{np}x_2^{np}x_1^{-np}x_2^{-np})^n:~ p \textrm{ prime }\}$$
is independent\footnote{Adyan uses the term ``irreducible'' in place of ``independent''.} (and in the book \cite{A79} Adyan proves the same statement for $n\ge 1001$).

This shows that $\F_\infty$ has $2^{\aleph_0}$ many fully invariant subgroups. Observe that for any natural number $r \ge 2$, $\F_\infty$ is isomorphic to the commutator subgroup of $\F_r$, the rank $r$ free group. Any fully invariant subgroup of the commutator subgroup of $\F_r$ is fully invariant as a subgroup of $\F_r$ because the commutator subgroup is itself fully invariant. 

\end{proof}

\subsection{The weak* topology}

Let $X$ be a compact metrizable space. Let $\cP(X)$ denote the set of all Borel probability measures on $X$. This set admits the {\em weak* topology}: a sequence $\{\mu_i\}_{i=1}^\infty \subset \cP(X)$ converges to $\mu_\infty \in \cP(X)$ if and only if: for every continuous function $f \in C(X)$ 
$$\lim_{n\to\infty } \int f~d\mu_i = \int f~d\mu_\infty.$$

By the Banach-Alaoglu Theorem, $\cP(X)$ is compact and metrizable. Moreover, it is a convex subspace of the Banach dual of $C(X)$. Recall that a {\em Choquet simplex} $\Delta$ is a compact convex subset of a locally convex topological vector space with the property that for every $\mu \in \Delta$ there exists a unique Borel probability measure $\theta$ on $\Delta^{e}$, the set of extreme points of $\Delta$, such that $\int \nu~d\theta(\nu) = \mu$. It is easy to see that $\cP(X)$ is a Choquet simplex. Indeed the map $\delta:X \to \cP(X)^e$ defined by $\delta(x)$ is the Dirac probability measure concentrated on $x$, is a homeomorphism. So for any $\mu \in \cP(X)$, $\mu  = \int \nu ~d\delta_*\mu(\nu)$ is the unique extremal decomposition.

Let $G$ be a topological group. An action of $G$ on $X$ is {\em continuous} if the map $(g,x) \mapsto gx$ (from $G\times X$ to $X$) is continuous. A Borel measure $\mu \in \cP(X)$ is {\em $G$-invariant} if $\mu(gA)=\mu(A)$ for every $g\in G$ and Borel $A \subset X$. Let $\cP_G(X) \subset \cP(X)$ denote the subspace of $G$-invariant measures. It is closed in $\cP(X)$ and is therefore compact in the weak* topology. It is also convex, so we may let $\cP^e_G(X) \subset \cP_G(X)$ denote the subspace of extreme points.  $\cP_G(X)$ is a Choquet simplex (see \cite{Ph01}).

 \subsection{Dynamics}\label{2.5}
Let $G$ be a topological group and $(X,\mu)$ a standard probability space. An action of $G$ on $X$ is {\em measurable} if the map $(g,x)\mapsto gx$ (from $G\times X$ to $X$) is measurable. It is {\em probability-measure-preserving} (pmp) if in addition to being measurable, $\mu(gA)=\mu(A)$ for every $g\in G$ and measurable $A \subset X$. In this case, $G$ acts by unitaries (shifts) on the function space $L^2(X)=L^2(X,\mu)$ via the Koopman representation: $g\phi(x):=\phi(g^{-1}x)$. 
 
 \begin{defn}\label{weak}
A pmp action $G \cc (X,\mu)$ is {\em weakly mixing} if $L^2(X)$ has no finite dimensional subspaces which are $G$-invariant, besides the zero subspace and the subspace of constant functions.
 \end{defn}
 
 \begin{defn}\label{ergodic}
A pmp action $G \cc (X,\mu)$ is {\em ergodic} if any $G$-invariant subset of $X$ is either null or conull.
 \end{defn}
We can also reformulate this definition in terms of unitary representations. Indeed, $G \cc (X,\mu)$ is ergodic iff $L^2(X)$ contains no non-constant $G$-invariant vectors. %Equivalently, any $G$-invariant vector in $L^2(X)$ is a constant function.

We list also the properties of weak mixing that we will use. First we remind the definition of a factor system.
\begin{defn}
An $G \cc (Y,\nu)$ is a {\em factor} of $G \cc (X,\mu)$ if there is a $G$-equivariant measurable map $p:X\to Y$ such that $p_*\mu=\nu$, where $p_*\mu$ is defined by $p_*\mu(B)=\mu(p^{-1}B)$ for every measurable $B\subset Y$. In this case $L^2(Y,\nu)$ isometrically and $G$-equivariantly embeds in $L^2(X,\mu)$, by the map $\phi\mapsto \phi\circ p$.
\end{defn}

\begin{lem}\label{properties} Let $G \cc (X,\mu)$ be a pmp action. 
\begin{itemize}
\item[a)] If $H$ is a subgroup of $G$ and the action $H\cc (X,\mu)$ is weakly mixing then $G\cc (X,\mu)$ is weakly mixing. 
\item[b)] A factor of a weakly mixing system is weakly mixing and a factor of an ergodic system is ergodic.
%\item[c)] If both actions $G\cc(X,\mu)$ and $G\cc(X',\mu')$ are weakly mixing then the diagonal action $G\cc(X\times X',\mu\otimes \mu')$ is also weakly mixing.
\end{itemize}
\end{lem}

 \section{Characteristic random subgroups of free groups}\label{sec:free}

 Let us recall a few definitions and set notation. We let $\F_r$ denote a free group of rank $r$ with $2\le r \le \aleph_0$. (Recall that the {\em rank} of a free group is the cardinality of a free generating set). Keep in mind that every subgroup of a free group is a free group. A subgroup $K\le \F_r$ is {\em characteristic} if $\phi(K)=K$ for every $\phi \in \Aut(\F_r)$.
 
In order to prove Theorem \ref{thm:free}, we will construct a CRS $\lambda$ of $\F_r$ for every characteristic subgroup $K\le \F_r$ which lies in the commutator subgroup (so $K \le  [\F_r,\F_r]$) in such a way that the normal closure $N_\lambda=K$. Theorem \ref{thm:Adyan} implies there are $2^{\aleph_0}$ many such characteristic subgroups and so this will imply Theorem \ref{thm:free}. %While the construction of $\mu_K$ from $K$ is simple, verifying that $\mu_K$ is $\F_r$-weakly mixing is nontrivial and accounts for most of the proof. It is no easier to prove that $\mu_K$ is $\F_r$-ergodic than to prove that $\mu_K$ is $\F_r$-weakly mixing.

%In order to prove that $\mu_K \ne \mu_L$ we will need the following definition:
%\begin{defn}[Normal closure]
%Let $\lambda$ be a conjugation-invariant measure on $\Sub(\F_r)$. The {\em normal closure of $\lambda$}, denoted $N_\lambda$, is the subgroup of $\F_r$ generated by all subgroups in the support of $\lambda$.
%\end{defn}

The next proposition contains the construction of $\lambda$. First we need a definition.
\begin{defn}\label{defn:hull}[HT]
 Let $G$ be a locally compact group and $\lambda \in \IRS(G)$. Recall that the {\em support} of $\lambda$ is the smallest closed subset $X \subset \Sub(G)$ such that $\lambda(X)=1$. The {\em hull} of $\lambda$ is the intersection of all subgroups in the support of $\lambda$. It is a subgroup of $G$, denoted by $H_\lambda$. The {\em normal closure} of $\lambda$ is the smallest closed subgroup of $G$ containing all of the subgroups in the support of $\lambda$. It is a subgroup of $G$ denoted by $N_\lambda$. %This definition is due to Hartman-Tamuz \cite{HT}.
 \end{defn}
 \begin{prop}\label{prop:key}
Let $K\le \F_r$ be an infinite rank characteristic subgroup and $p\ge 2$ be prime. Let $K_p := [K,K]K^p$ be the subgroup of $K$ generated by the commutator $[K,K]$ and the $p$-th powers of elements of $K$. Then there exists a continuous CRS $\lambda$ of $\F_r$ such that $K_p=H_\lambda\le N_\lambda=K$ where $H_\lambda,N_\lambda$ are the hull and normal closure of $\lambda$ (Definition \ref{defn:hull}). Moreover, if the action of $\F_r$ on $\widehat{K/K_p}$ $\cong \hcA_p$ is weakly mixing (with respect to Haar measure) then $\lambda$ is $\F_r$-weakly mixing. The action of $\F_r$ on $\widehat{K/K_p}$ is given by
$$(g \cdot \phi)(f K_p) = \phi(g^{-1}fg K_p)\quad \forall f\in K, g\in \F_r, \phi \in \widehat{K/K_p}.$$
\end{prop}

\begin{proof}
Because $K$ has infinite rank, $K/K_p$ is isomorphic with $\cA_p$ and therefore $\widehat{K/K_p} \cong \hcA_p \cong (\Z/p\Z)^\N$ is compact. Let $X$ be a Haar-random element of $\widehat{K/K_p}$. Because Haar measure is automorphism-invariant (by Halmos \cite{Ha43}), $\Ker(X)\le K/K_p$ is a CRS of $K/K_p$. Recall that $X$ is a homomorphism from $K/K_p$ to $\T$. We let $\Ker(X)$ denote its kernel and let $H\le K$ be the inverse image of $\Ker(X)$ under the quotient map $q:K \to K/K_p$ (so $H = q^{-1}(\Ker(X))$). Then $H$ is a CRS of $K$ and, since $K$ is characteristic, $H$ is a CRS of $\F_r$. Let $\lambda$ denote the law of $H$. 

We claim that $\lambda$ is continuous. Because every non-trivial element of $K/K_p$ has order $p$, $X$ has order $p$ a.s.  Therefore, $\Ker(X)$ has index $p$ in $K/K_p$. The automorphism group of $\cA_p=\oplus_\N (\Z/p\Z)$ acts transitively on the set of all index $p$ subgroups. Therefore, the orbit of $\Ker(X)$ under $\Aut(K/K_p)$ is infinite almost surely. This implies that $\lambda$ is non-atomic.

A subgroup $J\in \Sub(\F_r)$ is in the support of $\lambda$ if and only if $J\le K$ and $J$ has index $p$ in $K$. Because this collection of subgroups generates $K$, $N_\lambda=K$. The intersection of all such subgroups is $K_p=H_\lambda$.

Now suppose that the action of $\F_r$ on $\widehat{K/K_p}$ is weakly mixing. Let $\mu$ denote Haar measure on $\widehat{K/K_p}$ and observe that the map $q^{-1}\Ker:\widehat{K/K_p}\to \Sub(K) \subset \Sub(\F_r)$ is $\F_r$-equivariant (where $\F_r$ acts on $\Sub(K)$ by conjugation) and $q^{-1}\Ker_*\mu = \lambda$. Because weak mixing is preserved under factors, this implies $\lambda$ is $\F_r$-weakly mixing.
\end{proof}

%\begin{remark}
%The construction above also works when $p\ge 2$ is not prime but it is easier to see why $\lambda$ is continuous when $p$ is prime.
%\end{remark}

\begin{lem}\label{lem:mixing}
If $K\le \F_r$ is a characteristic subgroup and $K$ lies inside the commutator subgroup $[\F_r,\F_r]$ then for any prime $p$, the action of $\F_r$ on $\widehat{K/K_p}\cong \hcA_p$ is weakly mixing with respect to the Haar measure on $\widehat{K/K_p}$.
\end{lem}
We will prove this lemma in the next section. Let us see now how it implies Theorem \ref{thm:free}:

\begin{proof}[Proof of Theorem \ref{thm:free}]
By Theorem \ref{thm:Adyan} every nonabelian free group admits an uncountable family of characteristic subgroups. The commutator subgroup $[\F_r,\F_r]$ of $\F_r$ is an infinite rank free group. It is also characteristic. Because a characterisic subgroup of a characteristic subgroup is characteristic in the ambient group, Theorem \ref{thm:Adyan} implies that the set of all characteristic subgroups of $\F_r$ that lie in $[\F_r,\F_r]$ has cardinality $2^{\aleph_0}$. The theorem now follows immediately from Proposition \ref{prop:key} and Lemma \ref{lem:mixing}.
\end{proof}

 \subsection{Mixing}
 
 In this subsection we prove Lemma \ref{lem:mixing} after the next two lemmas.
 
 \begin{lem}\label{lem:Petersen}
 Let $G$ be a countable abelian group and $\hat{G}$ its Pontryagin dual. Let $\Gamma \le  \Aut(\hat{G})$ be a subgroup.  Then the action of $\Gamma$ on $\hat{G}$ is weakly mixing with respect to the Haar measure on $\hat{G}$ if and only the action of $\Gamma$ on $G$ has no finite orbits other than the trivial orbit containing the identity. The action of $\Gamma$ on $G$ is induced from the inclusion $\Gamma \le  \Aut(\hat{G})=\Aut(G)$ (see Proposition \ref{prop:Pontryagin} for the identification of $\Aut(\hat{G})$ with $\Aut(G)$).
 \end{lem}
 
 \begin{proof}
 This lemma is well-known. Probably the first time a version of it appeared is in \cite[Theorem 1]{Ha43}. There is also a version in \cite[Theorem 5.7]{Pe83}. For the sake of completeness we give the argument here. Let $\mu$ denote the normalized Haar measure on $\hat{G}$. This measure is $\Aut(\hat{G})$-invariant and therefore, $\Gamma$-invariant. For $g\in G$, let $\chi_g \in L^2(\hat{G},\mu)$ be the function $\chi_g(h) = \exp(2\pi i h(g))$. It is well-known that $\{\chi_g:~g\in G\}$ is an orthonormal basis of $L^2(\hat{G},\mu)$. 
 
Suppose that the action of $\Gamma$ on $G$ has no nontrivial finite orbits. Then for every nonzero $g\in G$ there is a sequence $\{\gamma_i\}_{i=1}^\infty \subset \Gamma$ such that $\lim_{i\to\infty} \gamma_i(g)=+\infty$ in the sense that for every finite subset $F \subset G$, $\gamma_i (g) \notin F$ for all sufficiently large $i$. By a diagonalization argument, there exists such a sequence $\{\gamma_i\}_{i=1}^\infty \subset \Gamma$ such that $\lim_{i\to\infty} \gamma_i(g)=+\infty$ for every nonzero $g\in G$.

Let $g,h \in G$ be nonzero elements. We let $\Gamma$ act on $L^2(\hat{G},\mu)$ through the Koopman representation.  Then 
$$\langle \gamma_i \chi_g, \chi_h\rangle = \langle \chi_{\gamma_i(g)}, \chi_h \rangle $$
tends to 0 as $i\to\infty$ where $\langle \cdot, \cdot \rangle$ denotes the inner product in $L^2(\hat{G},\mu)$. This is because $\gamma_i g$ is eventually not equal to $h$ and $\{\chi_g:~g\in G\}$ is an orthonormal basis. It follows that if $g,h \in G$ are arbitrary then
$$\lim_{i \to \infty} \langle \gamma_i \chi_g, \chi_h\rangle =\left( \int \chi_g~d\mu\right)\left(\int \chi_h~d\mu\right).$$
Because the span of $\{\chi_g:~g\in G\}$ is dense in $L^2(\hat{G},\mu)$, 
$$\lim_{i \to \infty} \langle \gamma_i f_1, f_2 \rangle =\left( \int f_1~d\mu\right)\left(\int f_2~d\mu\right)$$
for every $f_1,f_2 \in L^2(\hat{G},\mu)$. In other words, for every $f\in L^2(\hat{G},\mu)$, $\gamma_i f$ limits to the constant $\int f d\mu$ in the weak topology on $L^2(\hat{G},\mu)$ as $i\to\infty$. Since $\Gamma$ is acting unitarily on $L^2(\hat{G},\mu)$, this implies that the only finite-dimensional $\Gamma$-invariant subspace consists of the constants. So $\Gamma \cc \hat{G}$ is weakly mixing.

On the other hand, if $\{g_1,\ldots, g_n\}$ is a nontrivial finite orbit of the $\Gamma$-action on $G$ then $\chi_{g_1}+\cdots + \chi_{g_n}$ is a nonconstant $\Gamma$-invariant function on $\hat{G}$. Therefore the action of $\Gamma$ on $\hat{G}$ is non-ergodic. In particular, it is not weakly mixing.
 \end{proof}

\begin{notation}
For $g\in \F_r$, we let $\Ad(g) \in \Aut(\F_r)$ denote the inner automorphism defined by $\Ad(g)(x)=gxg^{-1}$. More generally, if $C \le  B$ are normal subgroups of $\F_r$ then we also let $\Ad(g) \in \Aut(B/C)$ be the automorphism $\Ad(g)(bC) = gbg^{-1}C$. 
\end{notation}

Recall that an element $g_0 \in \F_r$ is {\em primitive} if there exists a free generating set $S \subset \F_r$ that contains $g_0$.
\begin{lem}\label{lem:free-basis}
Let $K\le \F_r$ be an infinite rank normal subgroup. Suppose there is a primitive element $g_0\in \F_r$ such that $g_0K$ has infinite order in $\F_r/K$. Then there exists a free generating set $\cB$ for $K$  that is $\Ad(g_0)$-invariant. Moreover, every $\Ad(g_0)$-orbit in $\cB$ is infinite.
\end{lem}

\begin{proof}
Because $g_0$ is primitive there exists a free generating set $S \subset \F_r$ that contains $g_0$. Let $J$ be the subgroup of $\F_r$ generated by $K$ and $g_0$.  Let $\Sch(J\backslash \F_r,S)$ denote the Schreier right-coset graph of $J \backslash \F_r$: its vertex set is $J \backslash \F_r$ and for every coset $Jh \in J \backslash \F_r$ and every $s\in S$ there is an edge from $Jh$ to $Jhs$ labeled $s$. 

Recall that a subgraph of a graph is {\em spanning} if it contains all the vertices. It is a {\em tree} if it is simply-connected and it is a {\em forest} if each of its connected components is a tree. Because $\Sch(J\backslash \F_r,S)$ is a connected graph, there exists a spanning tree $T \subset \Sch(J\backslash \F_r,S)$. 

Now consider the Schreier right-coset graph $\Sch(K\backslash \F_r,S)$. The quotient map $\pi:K \backslash \F_r \to J \backslash \F_r$ determines a covering map from $\Sch(K\backslash \F_r,S)$ onto $\Sch(J\backslash \F_r,S)$, also denoted by $\pi$. Observe that $\pi^{-1}(T)$ is a spanning forest of $\Sch(K\backslash \F_r,S)$.

Let $T' \subset \Sch(K\backslash \F_r,S)$ be the union of $\pi^{-1}(T)$ with all edges of the form $\{Kg_0^m, Kg_0^{m+1}\}$ for $m \in \Z$. We claim that $T'$ is a spanning tree of $\Sch(K\backslash \F_r,S)$. It is spanning because $\pi^{-1}(T)$ is spanning. To see that $T'$ is connected, we will construct a path in $T'$ from an arbitrary coset $Kx \in K \backslash \F_r$ to the identity coset $K \in K \backslash \F_r$. Because $T$ is a spanning tree there exists a path in $T$ from $Jx$ to $J$. We observe that $\pi^{-1}(J) = \{Kg_0^m:~m\in \Z\}$. Therefore this path lifts to a path in $\pi^{-1}(T)$ from $Kx$ to $Kg_0^m$ for some $m\in \Z$. We may then append to this path all edges of the form $(Kg_0^i, Kg_0^{i+1})$ for $0\le i \le m-1$ (if $m\ge 0$) to obtain the required path. The case $m \le 0$ is similar.

To see that $T'$ is simply connected it suffices to show that the connected component of $\pi^{-1}(T)$ intersects the infinite path $p=\{ \{Kg_0^m, Kg_0^{m+1}\}:~m\in \Z\}$ in exactly one vertex. Indeed suppose that $Kg_0^n \ne Kg_0^m$ are in the same component of $\pi^{-1}(T)$. Let $q$ be a simple path from $Kg_0^n$ to $Kg_0^m$ in $\pi^{-1}(T)$. Since $\pi$ is a covering map, $\pi(q)$ is a simple path from $J=\pi(Kg_0^m)$ to $J=\pi(Kg_0^n)$. However, this contradicts that $\pi$ is a covering map and $T$ is a tree. This shows that $T'$ is a spanning tree as required.

Next we observe that $T'$ is invariant under left-multiplication by $g_0$ (here we are using the fact that $K$ is normal, so multiplication on the left is well-defined). This follows from the observation that 
$$\pi(g_0Kx)=\pi(Kg_0x)=Jx=\pi(Kx)$$
for any $x$ and the path $p$ is also left-$g_0$-invariant.

 Let $E$ be the set of all edges in $\Sch(K\backslash \F_r,S)$ that are not in $T'$. Because $T'$ is left-$g_0$-invariant, so is $E$. Choose an orientation for every edge in $E$ so that the left-action of $g_0$ on $E$ preserves orientations.

For every vertex $v$ of $\Sch(K\backslash \F_r,S)$, let $p_v$ be the unique oriented path in $T'$ from the identity coset $K$ to $v$.
For every oriented edge $e=(v,w) \in E$ let $p_e=p_v\cdot e\cdot p_w^{-1}$ denote the concatenation of $p_v$, $e$ and $p_w^{-1}$. Because $p_e$ is a circuit based at $K$, reading off its edge labels gives an element $k_e \in K$. Moreover, by a well-known result of Schreier, $\cB:=\{k_e:~ e\in E\}$ is a free basis for $K$. Observe that $k_{g_0e} = \Ad(g_0)(k_e)$ for any $e\in E$. Because $E$ is left-$g_0$-invariant, this basis is $\Ad(g_0)$-invariant. Because $Kg_0$ has infinite order in $K\backslash \F_r$, it follows that for every $e \in E$, $\{g_0^ne\}_{n\in \Z}$ is infinite. Therefore, every $\Ad(g_0)$-orbit in $\cB=\{k_e:~e\in E\}$ is also infinite.
\end{proof}

\begin{proof}[Proof of Lemma \ref{lem:mixing}]
Let $K\le \F_r$ be a characteristic subgroup that lies in the commutator $[\F_r,\F_r]$. Then there exists a primitive element $g_0 \in \F_r$ such that $g_0K$ has infinite order in $\F_r/K$. For instance, any element of a basis of $\F_r$ satisfies this property. By Lemma \ref{lem:free-basis} there exists a $\Ad(g_0)$-invariant free basis $\cB$ for $K$ such that every $\Ad(g_0)$-orbit in $\cB$ is infinite. 

Let $hK_p \in K/K_p$ be a nonidentity element. By Lemma \ref{lem:Petersen}, it suffices to prove that $\{f hf^{-1}K_p:~f\in \F_r\}$ is infinite. Because $\cB$ is a free basis, $h=b_1^{e_1}\cdots b_n^{e_n}$ for some $b_1,\ldots, b_n \in \cB$ and exponents $e_i \in \Z \setminus \{0\}$ such that $b_{i+1} \ne b_i$ for $1\le i \le n-1$. Then
$$\Ad(g_0)^m(h) = (\Ad(g_0)^m(b_1))^{e_1}\cdots (\Ad(g_0)^m(b_n))^{e_n}.$$
Because the basis $\cB$ is $\Ad(g_0)$-invariant, each nontrivial $\Ad(g_0)$-orbit in $\cB$ is infinite, and the image $\bar{\cB}$ of $\cB$ in $K/K_p$ is a basis for the elementary abelian $p$-group $K/K_p$, it follows that $\{\Ad(g_0)^m(h)K_p:~m\in \Z\} \subset K/K_p$ is infinite. This proves the lemma.

\end{proof}

  \section{Random homomorphisms}\label{sec:homom}

Suppose $G$ is a countable abelian group and $K$ is a compact abelian group. Let $\Hom(G,K)$ denote the space of all homomorphisms from $G$ to $K$ with the topology of pointwise convergence. This set naturally embeds into the product space $K^G$ as a closed set. By Tychonoff's Theorem, $K^G$ is compact. Therefore, $\Hom(G,K)$ is compact. It is also an abelian group under pointwise addition and the group structure is compatible with its topology. For any subgroup $H\le G$ and closed subgroup $L\le K$ there is a natural embedding of $\Hom(G/H,L)$ into $\Hom(G,K)$. Indeed, we may identify $\Hom(G/H,L)$ with the set of homorphisms of $\phi:G \to K$ such that $\phi(G) \le L$ and $H\le \Ker(\phi)$. Also $\Hom(G/H,L)$ is a closed subgroup of $\Hom(G,K)$. 

The group $\Aut(K)$ acts on $\Hom(G,K)$ by 
$$(\psi_* h)(g) = \psi(h(g))\quad \forall \psi \in \Aut(K), h\in \Hom(G,K), g\in G.$$
A key step in the proof of Theorems \ref{thm:hcA}-\ref{thm:cAn} is the next result: 
\begin{thm}\label{thm:key2}
Let $G$ be a countable group and let $\chi$ be an $\Aut(\hcA)$-invariant and indecomposable Borel probability measure on $\Hom(G,\hcA)$. Then  there exist a subgroup $H\le G$ such that $\chi$ is the normalized Haar measure on $\Hom(G/H,\hcA)$.
\end{thm}

\begin{proof}
We identify $\hcA$ with the infinite dimensional torus $\T^\N$ and $\Hom(G,\hcA)$ with the infinite product $\Hom(G,\T)^\N$.  Let $h =(h_1,h_2,\ldots) \in \Hom(G,\T)^\N$ be a random homomorphism with law $\chi$. If $\sigma:\N \to \N$ is an arbitrary permutation and $\rho_\sigma \in \Aut(\hcA)$ is defined by $\rho_\sigma(x_1,x_2,\ldots) = (x_{\sigma(1)},x_{\sigma(2)},\ldots)$ then $(\rho_\sigma)_*h = (h_{\sigma(1)}, h_{\sigma(2)},\ldots)$. This means that $h$ is an exchangeable sequence of random variables. So the de Finetti-Hewitt-Savage Theorem\footnote{This Theorem states that if $X_1,X_2,\ldots$ is a family of random variables each taking values in a compact metrizable space and whose joint law is invariant under all permutations of the indices then $X_1,X_2,\ldots$ is conditionally iid.} (see \cite{G} for example) implies that the variables $\{h_i\}_{i\in \N}$ are conditionally iid (independent identically distributed). This means the following. Let $\cP(\Hom(G,\T))$ denote the space of all Borel probability measures on $\Hom(G,\T)$ with the weak* topology. Then there exists a unique Borel probability measure $\omega$ on $\cP(\Hom(G,\T))$ such that
$$\chi = \int \mu^\N~d\omega(\mu)$$
where, for any $\mu \in \cP(\Hom(G,\T))$, $\mu^\N$ denotes the product measure on $\Hom(G,\T)^\N$. 

%Moreover, $\omega$ is unique. This is because the law of $h_1$ is $\int \mu~d\omega(\mu)$ and $\cP(\Hom(G,\T))$ is a Choquet simplex.

So there exists a random measure $\mu \in \cP(\Hom(G,\T))$ with law $\omega$ such that $\{h_i\}$ is an iid sequence of random variables with law $\mu$.

Define $\phi,\psi \in \Aut(\T^\N)$ by
\begin{eqnarray*}
\phi(x_1,x_2,x_3,\ldots) &=& (-x_1,x_2,x_3, \ldots) \\
\psi(x_1,x_2,x_3,\ldots) &=& (x_1+x_2, x_2, x_3, \ldots).
\end{eqnarray*}
%Let $\pi:\Hom(G,\T)^\N \to \Hom(G,\T)$ denote projection to the first coordinate. Let $
Observe that 
\begin{eqnarray*}
\phi_*(h_1,h_2,h_3,\ldots) &=& (-h_1,h_2,h_3, \ldots) \\
\psi_*(h_1,h_2,h_3,\ldots) &=& (h_1+h_2, h_2, h_3, \ldots).
\end{eqnarray*}
Because $\chi$ is invariant under $\phi_*$, we can repeat the argument above with the automorphisms $\rho_\sigma$ to conclude that the variables $-h_1,h_2,h_3,\ldots$ are iid with law $\mu'$ for some $\mu' \in \cP(\Hom(G,\T))$. Because we know that $h_2,h_3,\ldots$ are iid with law $\mu$, it follows that $\mu'=\mu$. Thus if $M:\Hom(G,\T) \to \Hom(G,\T)$ is the map $M(k)=-k$ then $M_*\mu=\mu$ for $\omega$-a.e. $\mu$.

Similarly, because $\chi$ is invariant under $\psi_*$ the variables $h_1+h_2,h_2,h_3,\ldots$ are iid with law $\mu''$ for some $\mu'' \in \cP(\Hom(G,\T))$. Because we know that $h_2,h_3,\ldots$ are iid with law $\mu$, it follows that $\mu''=\mu$. For $z \in \Hom(G,\T)$, let $A_z:\Hom(G,\T) \to \Hom(G,\T)$ be the addition map $A_z(k)=k+z$. Because $h_1+h_2,h_2$ have law $\mu$ and $h_1+h_2$ is independent of $h_2$ it follows that $(A_z)_*\mu=\mu$ for $\mu$-a.e. $z \in \Hom(G,\T)$ and for $\omega$-a.e. $\mu \in \cP(\Hom(G,\T))$. 

Let $S_\mu$ be the support of $\mu$ and $S'_\mu$ be the set of all $z \in S_\mu$ such that $(A_z)_*\mu=\mu$. We have already shown that $\mu(S'_\mu)=1$. So $S'_\mu$ is dense in $S_\mu$. Thus for any $z \in S_\mu$ there exists a sequence $\{z_n\}_{n=1}^\infty \subset S'_\mu$ such that $z_n \to z$ as $n \to\infty$. It follows that $A_{z_n}$ converges to $A_z$ uniformly. So $\lim_{n\to\infty} (A_{z_n})_*\mu = (A_z)_*\mu$. Because $(A_{z_n})_*\mu=\mu$ for all $n$, this implies $(A_z)_*\mu=\mu$. In particular, the support of $\mu$ is invariant under the inverses map and addition. So $S_\mu$ is a closed subgroup of $\Hom(G,\T)$. Moreover, because $\mu$ is invariant under addition by elements of $S_\mu$, $\mu$ must be the Haar measure of $S_\mu$. %So $\mu^\N$ is the Haar measure of $S_\mu^\N$. We observe that $S_\mu^\N$ is $\Aut(\hcA)$-invariant (no matter what $S_\mu$ is). It follows

By definition $\Hom(G,\T) = \hat{G}$. Let $H_\mu=\Ann(S_\mu)\le G$. Then $S_\mu=\Hom(G/H_\mu,\T)$ and $\mu^\N$ is the Haar measure on $S^\N_\mu=\Hom(G/H_\mu,\T)^\N$ which we identify with $\Hom(G/H_\mu,\T^\N)$. For any subgroup $H\le G$, $\Hom(G/H,\T^\N)$ is $\Aut(\hcA)$-invariant. Because the measure $\omega$ is uniquely determined by $\chi$ and $\chi$ is an arbitrary $\Aut(\hcA)$-invariant probability measure,  it follows that the Haar measure on $\Hom(G/H,\T^\N)$ is $\Aut(\hcA)$-indecomposable for every $H\le G$. Because $\chi$ is $\Aut(\hcA)$-indecomposable it follows that $\chi$ must actually be the Haar measure on a subgroup of the form $\Hom(G/H, \T^\N)$.
\end{proof}

%\begin{lem}\label{l0}
%Let $\eta$ denote the Haar measure on $\hcA$. Then for $\eta$-a.e. $h\in \hcA$, the subgroup generated by $h$ is dense in $\hcA$.
%\end{lem}
%\begin{proof}
%Observe that the subgroup generated by $h$ is dense in $\hcA$ if and only if $\Ker(h)=\{0\}$. Since $\A$ is countable, it suffices to check that for each nonzero $x\in \A$, the subgroup $\Ann (x)=\{h\in \hcA: h(x)=0 \}$ has zero Haar measure. But the map $h \mapsto h(x)$ is a homomorphism of $\hcA$ into $\T$ and therefore $\Ann(x)$ must have infinite index in $\hcA$. Since every subgroup of infinite index has zero Haar measure, $\eta(\Ann(x))=0$.
%\end{proof}

\section{Characteristic random subgroups of abelian groups}\label{sec:abelian}

In this section we prove Theorems \ref{thm:hcA}-\ref{thm:cAn}. But first we need two lemmas and a proposition. We note that this lemma in other words means that a divisible abelian group is an injective module over $\Z$, which is well-known. We include the proof for convenience.

\begin{defn}
An abelian group $X$ is {\em divisible} if for every $x\in X$ and integer $n\ne 0$ there exists $y\in X$ such that $ny=x$. For example, $\T$ and $\T^\N$ are divisible.
\end{defn}

\begin{lem}\label{lem:extension}
Let $G$ be a countable abelian group. Also let $X$ be a divisible abelian group. Let $G_0\le G$ be a subgroup and $h:G_0 \to X$ a homomorphism. Then there exists a homomorphism $h':G \to X$ such that $h'(g)=h(g)$ for all $g\in G_0$.
\end{lem}

\begin{proof}
By induction, it suffices to prove that if $g_1 \in G \setminus G_0$ and $G_1\le G$ is the subgroup generated by $G_0 \cup \{g_1\}$ then there exists a homomorphism $h_1:G_1 \to \T$ such that $h_1(g)=h(g)$ for all $g\in G_0$.

Suppose that $n g_{1} \notin G_0$ for any nonzero integer $n$. Then define $h_1:G_{1} \to X$ by
$$h_1(r+ng_{1})=h(r)$$
for any $r\in G_0$ and any integer $n$. Observe that if $r+ng_{1} = s+mg_{1}$ for some $r,s \in G_0$ and $n,m \in \Z$ then $(r-s)=(m-n)g_{1}$ implies $m=n$ and $r=s$. So $h_1$ is well-defined and is a homomorphism.

Suppose now that there exists a positive integer $p_0$ such that $p_0g_{1} = g_0 \in G_0$. Let us assume that $p_0$ is the smallest positive integer such that $p_0g_{1} \in G_0$. Let $t \in X$ be an element such that $p_0t=h(g_0)$. 

\noindent {\bf Claim}. If $r,s \in G_0$ and $m,n \in \Z$ satisfy $r+ng_{1}=s+mg_{1}$ then $h(r)+nt=h(s)+mt$. 

\begin{proof}[Proof of Claim]
Because $(m-n)g_{1} = r-s \in G_0$, the definition of $g_0$ implies $m-n = kp_0$ for some integer $k$. So $r-s=(m-n)g_{1} = kp_0g_{1}=kg_0$. So
$$h(r-s)=h(kg_0)=kh(g_0)=kp_0t = (m-n)t.$$ %This proves the Claim.
\end{proof}
We now define
$$h_1(r+ng_{1})=h(r) + nt$$
for any $r\in G_0$ and any integer $n$. By the claim this is well-defined and is a homomorphism. 
\end{proof}

\begin{defn}
Define $\overline{\Image}:\Hom(G,\hcA) \to \Sub(\hcA)$ by: $\overline{\Image}(h)$ is the closure of the image of $h$.
\end{defn}
%We need the following proposition.
%We will prove this in the next section. 
\begin{prop}\label{prop:haar}
Let $G$ be a countable abelian group, $\chi$ denote the Haar measure on $\Hom(G,\hcA)$ and $h\in \Hom(G,\hcA)$ be a random homomorphism with law $\chi$. %For $h\in \Hom(G,\cA)$, let $\overline{\Image}(h) \in \Sub(\hcA)$ denote the closure of the image of $h$. % and $\overline{\Image}:\Hom(G,\hcA) \to \Sub(\hcA)$ the map $\overline{\Image}(h)$ is the closure of the image of $h$. Then
\begin{itemize}
\item If $G$ is finite then $h$ is injective a.s.
\item If $G=\cA_n$ for some $n \ge 2$ then $\overline{\Image}(h) = \Ann(n\cA)$ a.s. 
\item If $G$ contains an element of infinite order then $\overline{\Image}(h) = \hcA$ a.s. 
\item If the set of orders of elements of $G$ is unbounded then $\overline{\Image}(h) = \hcA$ a.s. 
\end{itemize}

\end{prop}

\begin{proof}

Suppose $G$ is finite. Because $\chi$ is $\Aut(\hcA)$-invariant and ergodic, there exists a subgroup $H\le G$ such that $\Ker(h)=H$ for $\chi$-a.e. $h$. Let $h' \in \Hom(G,\hcA)$ be an injective homomorphism. Because $\chi$ is the Haar measure on $\Hom(G,\hcA)$, $\chi$ is invariant under the map $h \mapsto h+h'$. Therefore, $H=\Ker(h)=\Ker(h+h')$ for $\chi$-a.e. $h$. Because the kernel of $h'$ is trivial, this implies $H$ is trivial. This proves the first item.

Suppose now that $G=\cA_n = \oplus_\Z (\Z/n\Z)$. Then for $\chi$-a.e. $h\in \Hom(G,\hcA)$, the image of $h$ lies in the subgroup $\Ann(n\cA)$ (this is because $\Ann(n\cA)$ contains every element of order $n$ in $\hcA$). Let $X_1,X_2,\ldots \in \hcA$ be iid random variables each with law equal to the Haar measure on $\Ann(n\cA)$. Then the law of the subgroup $\overline{\langle X_1,X_2,\ldots \rangle}$ is the same as the law of $\overline{\Image}(h)$ where $h \in \Hom(\cA_n,\hcA)$ is chosen uniformly at random. So it suffices to show that $\overline{\langle X_1,X_2,\ldots \rangle}  = \Ann(n\cA)$ a.s. By duality this is equivalent to showing that $\cap_{i=1}^\infty \Ker(X_i) = n\cA$ almost surely. We have that  $\cap_{i=1}^\infty \Ker(X_i) \supset n\cA$, and will now show the inverse inclusion.
Note that since $X_i$ are iid, we have that $$\mathbb{P}(v\in\cap_i \Ker(X_i))=\lim_i \mathbb{P}(X_1(v)=0)^i,$$ which is $0$ if $\mathbb{P}(X_1(v)=0)<1$. Thus it suffices to show that $\mathbb{P}(X_1(v)=0)=1$ implies that $v\in n\A$. However $\mathbb{P}(X_1(v)=0)=1$ says that $X(v)=0$ for a.a. $X\in\Ann(n\A)$ with respect to Haar measure on $\Ann(n\A)$. Since the condition $X(v)=0$ defines a closed set of $X$'s, we have by continuity that $X(v)=0$ for all $X\in\Ann(n\A)$. Thus $v\in\Ker(\Ann(n\A))=n\A$ (see \S \ref{sec:duality}).

To prove the last two items, note first that $\overline{\Image}(h) = \hcA$ is equivalent to $\Ker(h(G))=0$. So it suffices to show that $\Ker(h(G))=0$ almost surely. Because $\A$ is countable, this is equivalent to the statement that for any nonzero $v\in\A$,  $v \notin \Ker(h(G))$ almost surely. Let us identify $\A$ with $\oplus_\N \Z$. Since $\chi$ is $\Aut(\hcA)$-invariant, it suffices to show check this condition in the special case $v=ke_1$ where $k \geq 1$ and $e_1=(1,0,\dots)$.

For any $g\in G$, the probability that $ke_1 \in \Ker(h(G))$ is at most the probability that $ke_1 \in \Ker(h(g))$. So it suffices to show there is a sequence $\{g_n\} \subset G$ such that $\mathbb{P}(h(g_n)(ke_1) =0) \to 0$ as $n\to\infty$ where $\mathbb{P}(\cdot)$ denotes probability.

Fix $g\in G$. Consider the map from $\Hom(G,\hcA)$ to $\hcA$ given by $h\mapsto h(g)$. This map is a continuous homomorphism. So it pushes the Haar measure on $\Hom(G,\hcA)$ forward to Haar measure, denoted by $\eta$, on the image subgroup $\{h(g):~h\in \Hom(G,\hcA)\}$. By Lemma \ref{lem:extension}, the image subgroup is either $\hcA$ (if $g$ has infinite order) or $\Ann(m\cA)$, if $g$ has order $m<\infty$. %In either case, it suffices to show that $\eta(\{x\in \hcA:~x(ke_1)=0\})=0$. 

If $g$ has infinite order then because the closed subgroup $\{x\in \hcA:~x(ke_1)=0\}$ has infinite index in $\hcA$, $\eta(\{x\in \hcA:~x(ke_1)=0\})=0$. Equivalently,  $\mathbb{P}(h(g)(ke_1) =0) =0$. Thus $\overline{\Image}(h)=\hcA$.

Suppose $g$ has finite order $m$. Identify $\Ann(m\A)$ with $\prod_\N \Z[1/m]/\Z$. We compute that
\[
\mathbb{P}(h(g)(ke_1)=0) = \eta(\{x| x(ke_1)=0\})=\frac{\#\{x\in \Z[1/m]/\Z \,\,|\,\,xk=0 \}}{m}=\frac{\gcd(k,m)}{m}.
\]
Since this ratio goes to $0$ when $m$ goes to infinity, we are done.
\end{proof}

We have the following obvious lemma, which is easily proved using uniqueness in the structure theorem for finite abelian groups.
\begin{lem}\label{lem:over}
%Let $F=\oplus_{j=1}^{s}\Z/q_j^{t_j}\Z$ with none $q_j^{t_j}$ dividing $n$, then $F$ is over $n$ (see the Definition \ref{defn:over}). 
Given a finite abelian group $F$ and integer $n\ge 1$, let $F_{(n)}=\{x\in F\,|\, nx=0\}$. Then for any finite abelian group $H$ there is a finite abelian group $F$ such that $F/F_{(n)} \cong H$ and $F$ is over $n$ (in the sense of Definition \ref{defn:over}). Moreover $F$ is uniquely determined up to isomorphism.
\end{lem}
\begin{proof}
%Let $p_1,\ldots, p_s$ be prime numbers and $u_1,\ldots, u_s$ be positive integers such that $n=p_1^{u_1}\cdots p_s^{u_s}$. 

By the classification of finite abelian groups, there are primes $p_1,\ldots, p_s$ and positive integers $t_1,\ldots, t_s$ such that $H=\oplus_{j=1}^{s}\Z/p_j^{t_j}\Z$. 

Let $F$ be any finite abelian group. Without loss of generality, $F = \oplus_{i=1}^r \Z/q_i^{u_i}\Z$ for some primes $q_i$ and integers $u_i \ge 1$ that are uniquely determined by $F$. Then
$$F_{(n)} = \bigoplus_{i=1}^r \frac{\textrm{lcm}(n,q_i^{u_i})}{n} \Z \Big/q_i^{u_i}\Z, \quad F/F_{(n)} \cong \Z\Big/\frac{\textrm{lcm}(n,q_i^{u_i})}{n} \Z.$$
Therefore, $F$ is over $n$ if and only if $q_i^{u_i}$ does not divide $n$ for every $i$. Suppose this is the case. Then $F/F_{(n)} \cong H$ if and only  if: after permuting indices if necessary, $r=s$, $p_i=q_i$ for all $i$ and $p_i^{t_i} =  \textrm{lcm}(n,p_i^{u_i})/n$ for all $i$. This condition uniquely determines $u_i$. Indeed, $u_i = t_i+k_i$ where $k_i \ge 0$ is determined by: $p_i^{k_i} \mid n$ and $p_i^{k_i+1} \nmid n$. 
\end{proof}

%\begin{lem}
%Let $X$ be a compact topological space $G$ be a group that acts by homeomorphisms on $X$. Denote by $\P_G(X)$ the convex closed set of all $G$-invariant measures on $X$. Denote by $\P^e_G(X)$ the set of extreme points of $\P_G(X)$.  Suppose that $Y\subset X$ is a $G$-invariant closed subset. Then $\P_G(Y)\subset \P_G(X)$ and $\P^e_G(Y)\subset \P^e_G(X)$.
%\end{lem}
%\begin{proof}
%The firs assertion is obvious. For the second, suppose that $\mu\in \P^e_G(Y)$, and $\mu=t\nu_1+(1-t)\nu_2$ for some $t\in (0,1)$ and $\nu_1,\nu_2\in \P_G(X)$. We have since $t\neq 0,1$ that the support of $\mu$ is the union of supports of $\nu_1$ and $\nu_2$. Since the support of $\mu$ is a subset of $Y$, it follows that the supports of $\nu_i$ are subsets of $Y$, therefore $\nu_1,\nu_2$ are in fact elements of $\P_G(Y)$. A contradiction.
%\end{proof}

\begin{proof}[Proof of Theorem \ref{thm:hcA}]
%We will prove Theorem \ref{thm:--} first. 
Observe that the special case $G=\hcA_n$ ($n\ge 2$) follows from the case $G=\hcA$ because $\hcA_n$ is naturally a characteristic subgroup of $\hcA$. Indeed we may regard $\hcA_n \cong (\Z[1/n]/\Z)^\N$ as a characteristic subgroup of $\hcA \cong \T^\N$. Therefore every CRS of $\hcA_n$ is automatically a CRS of $\hcA$.  In fact the inclusion map $\hcA_n \to \hcA$ induces an $\Aut(\hcA)$-equivariant inclusion map $\Sub(\hcA_n) \to \Sub(\hcA)$ which induces an affine embedding $\CRS(\hcA_n) \to \CRS(\hcA)$. Therefore it suffices to prove the special case $G=\hcA$.

Let $\lambda \in \CRS^e(\hcA)$. Define a measure $\chi$ on $\Hom(\cA,\hcA)$ by
$$\chi = \int \chi_{K}~d\lambda(K)$$
where $\chi_K$ denotes the Haar measure on the subgroup $\Hom(\cA,K) \le  \Hom(\cA,\hcA)$. Observe that $\overline{\Image}_*\chi=\lambda$. Because $\lambda$ is $\Aut(\hcA)$-invariant, so is $\chi$. By Theorem \ref{thm:key2}, $\chi$ is a convex integral of Haar measures on subgroups of $\Hom(\cA,\hcA)$ of the form $\Hom(\cA/H,\hcA)$. So $\lambda$ is a convex integral of measures of the form $\overline{\Image}_*\chi_{\cA/H}$ where $\chi_{\cA/H}$ denotes Haar measure on $\Hom(\cA/H,\hcA)$. Because $\lambda$ is indecomposable and each measure of the form $\overline{\Image}_*\chi_{\cA/H}$ is $\Aut(\hcA)$-invariant there is a countable abelian group $G$ such that $\lambda=\overline{\Image}_*\chi_G$ where $\chi_G$ denotes Haar measure on $\Hom(G,\hcA)$.

%Define $\overline{\Image}:\Hom(\cA,\hcA) \to \Sub(\hcA)$ by: $\overline{\Image}(h)$ is the closure of the image of $h$. 

%So it suffices to classify $\overline{\Image}_* \chi_{\cA/H}$ for every subgroup $H\le \cA$ and closed characteristic subgroup $L\le \hcA$. Indeed, every measure of the form $\overline{\Image}_* \chi_{\cA/H,L}$ is $\Aut(\hcA)$-invariant since $\overline{\Image}$ is $\Aut(\hcA)$-equivariant. 

By Proposition \ref{prop:haar} if there does not exist a finite bound on the order of elements of $G$ then the image of $h$ is dense in $\hcA$ for $\chi_G$-a.e. $h$. Thus  $\lambda=\delta_{\hcA}$, the point measure on $\hcA$. This measure corresponds to the pair $(0,[0])$, where $0$ is a trivial group (recall the statement of Theorem \ref{thm:hcA}).

Suppose now that there does exist a bound on the order of elements of $G$. Then the first Pr\"ufer Theorem implies 
$$G\cong F \oplus \cA_{n_1} \oplus \cdots \oplus \cA_{n_r}$$
for some finite abelian group $F$ and integers $n_1,\ldots, n_r \ge 1$ (note $\cA_1$ is the trivial group). Without loss of generality, we identify $G$ with the direct sum above. So for any $h \in \Hom(G,\hcA)$, 
$$\overline{h(G)} = \overline{h(F)} + \overline{h(\cA_{n_1})}+ \cdots + \overline{h(\cA_{n_r})}.$$
By Proposition \ref{prop:haar} we have that $\sum_i \overline{h(\A_{n_i})}=\sum_i \Ann(n_i\cA) = \Ann(m\A)$ for $\chi_G$-a.e. $h$ where $m=\textrm{lcm}(n_i)$.

If $F=F_0\oplus F_1$ for some subgroups $F_0,F_1$ and $mF_0 = 0$ then $h(F_0) \le \Ann(m\cA)$. Therefore 
$$\overline{h(G)} = h(F_1) + \Ann(m\cA) = \overline{h(F_1 \oplus  \cA_{n_1} \oplus \cdots \oplus \cA_{n_r})}$$
for $\chi_G$-a.e. $h$. So without loss of generality, we may assume that $F$ does not have any direct summand $F_0$ with $mF_0=0$. In other words, we may assume $F$ is over $m$ (Definition \ref{defn:over}). Thus $\lambda$ corresponds to the pair $(m,[F])$ in the notation of Theorem \ref{thm:hcA}. Now we only have to check that the pair $(m,[F])$ is uniquely determined by $\lambda$.

Because $\overline{h(G)}$ is a finite extension of $\Ann(m\cA)$ for a.e. $h$, the integer $m$ is uniquely determined by $\lambda$. By Proposition \ref{prop:haar}, $h$ restricted to $F$ is injective for $\chi_G$-a.e. $h$. So $F/F_{(m)} \cong \overline{h(G)}/\Ann(m\cA)$ for $\chi_G$-a.e. $h$. Lemma \ref{lem:over} now implies the isomorphism class of $F$ is uniquely determined by $\lambda$ and the requirement that $F$ is over $m$.

\end{proof}

\begin{proof}[Proof of Theorem \ref{thm:cAn}]
%Define $\Ann:\Sub(\cA) \to \Sub(\hcA)$ and $\Ker:\Sub(\hcA) \to \Sub(\cA)$ by
%$$\Ann(X) = \{ h \in \hcA:~ h(x)=0~ \forall x\in X\}$$
%$$\Ker(H) = \{x\in \cA:~h(x)=0~\forall h \in H\}.$$
%Observe that these maps are continuous inverses of each other and they are $\Aut(\cA)$-equivariant. So they induce homeomorphisms $\Ann:\Char(\cA) \to \Char(\hcA)$, $\Ker:\Char(\hcA) \to \Char(\cA)$ and affine homeomorphisms $\Ann_*: \CRS(\cA) \to \CRS(\hcA)$ and $\Ker_*:\CRS(\hcA) \to \CRS(\cA)$. Also $\Ann$ and $\Ker$ are order-reversing in the sense that $H\le K$ implies $\Ann(H) > \Ann(K)$ and similarly with $\Ker$. Lastly, they take addition to intersection and vice versa. To be precise, suppose $H,K$ are subgroups of  $\cA$. Then
%$$\Ann(H+K) = \Ann(H) \cap \Ann(K) \quad \Ann(H \cap K) = \Ann(H)+\Ann(K).$$
%A similar statement holds with $\Ker$ in place of $\Ann$.

Consider first the case $G=\cA$. As discussed in \S \ref{sec:duality}, any indecomposable CRS of $\cA$ is of the form $\eta=\Ker_*\lambda$ for some $\lambda \in \CRS^e(\hcA)$. So Theorem \ref{thm:hcA} implies $\lambda$ corresponds to the pair $(m,[F])$, where $m\ge 0$ is an integer and $F$ is a finite abelian group over $m$. Thus if $h \in \Hom(F,\hcA)$ is random with law equal to the Haar measure of $\Hom(F,\hcA)$ then $\eta$ is the law of $\Ker(h(F) + \Ann(m\cA))$. By (\ref{eqn:dual}) from \S \ref{sec:duality}, $\eta$ is the law of 
$$\Ker(h(F)) \cap \Ker(\Ann(m\cA))=\Ker(h(F)) \cap m\cA.$$
Note now that $v\in \Ker(h(F))$ if and only if for every $f\in F$ we have $0=h(f)(v)=\hat{h}(v)(f)$, where $\hat{h}:\A\to\hat{F}$ is the homomorphism dual to $h$. Since this is true for every $f\in F$, it follows that $\hat{h}(v)=0$, or $v\in \Ker(\hat{h})$. Thus $\Ker(h(F))=\Ker(\hat{h})$. To finish the proof in the case $G=\A$, it is left to note that the duality gives a continuous group isomorphism between $\Hom(F,\hcA)$ and $\Hom(\A,\hat{F})$, and that  $F\simeq\hat{F}$ since $F$ is finite.

Suppose now that $G=\A_n=\A/n\A$.  Since $n\A$ is a characteristic subgroup of $\A$, any CRS on $\A_n$ by taking its preimage under the factor map $\A\to \A/n\A=\A_n$ gives a CRS on $\A$, which will contain $n\A$ almost surely. Thus in order to describe $\CRS^e(\A_n)$ it suffices to describe the set of those elements $\CRS^e(\A)$ that contain $n\A$ with probability $1$. Clearly, they are those elements of $\CRS^e(\A)$ that correspond to the pair $(m,[F])$ with $m|n$ and $nF=0$.

Note that if $m|n$ and $nF=0$, then $m\A/n\A=m\A_n$ and $\Hom(\A,F)=\Hom(\A_n,F)$, since for any $h\in\Hom(\A,F)$ we have that $h(n\A)=nh(\A)\subset nF=0$. Thus an element in $\CRS^e(\A_n)$ corresponding to the pair $(n,[F])$ can be written as $m\A_n\cap\Ker(h)$, where $h\in \Hom(\A_n,F)$ has the Haar law.
\end{proof}

\subsection{Topology of $\CRS^e(\hcA)$}
Next we describe the topology of the space $\CRS^e(\hcA)$, identified with the set of pairs $(n,[F])$ such that $n\geq 0$ and $[F]$ is an isomorphism class of a finite abelian group $F$ over $n$.
\begin{thm}\label{thm:topology}
Let $(n_i,[F_i])$ be a sequence corresponding to $\lambda_i \in \CRS^e(\hcA)$. The following statements hold.
\begin{itemize}
\item If $n_i\to\infty$, then $\lambda_i\to \delta_{\hcA}$.
\item If $n_i$ stabilizes, but $\maxorder(F_i)$, the maximal order of elements in $F_i$, goes to $\infty$, then $\lambda_i\to \delta_{\hcA}$.
\item Suppose $n_i$ stabilizes to $n$. Suppose also that eventually $F_i\cong F\oplus (\oplus_{j=1}^s(\Z/q_j^{t_j}\Z)^{m_j(i)})$ for some fixed finite abelian group $F$, integer $s \ge 0$, primes $q_j$ and integers $t_j \ge 0$, and that for all $j$ $\lim_{i\to\infty} m_j(i)=\infty$ (this is the case of bounded $\maxorder(F_i)$). Then $\lambda_i$ converges to indecomposable CRS corresponding to the pair $(\textrm{lcm}(n,\{q_j^{t_j}\}_{j=1}^s),[F])$
\item For any sequence it is possible to choose subsequence which is one of the first three types.
\end{itemize}
\end{thm}

\begin{cor}\label{cor:top}
$\CRS^e(\hcA)$ is closed in $\CRS(\hcA)$ (i.e., $\CRS(\hcA)$ is a Bauer simplex).
\end{cor}

\begin{proof}[Proof of Theorem \ref{thm:topology}]
Recall (\cite{BHK09}) that for a compact group $K$, given a neighborhood $U\subset K$ of identity, we obtain a neighborhood $N_U(K_1)$ of $K_1\in\Sub(K)$ in the Chabauty topology  by 
\[
N_U(K_1):=\{H\leq K\,|\,H+U\supset K_1,\, K_1+U\supset H \}.
\]
Note that for any subgroup $T\in\Sub(K)$, if  $H\in N_U(K_1)$, then $H+T\in N_{U}(K_1+T)$.

Given $\epsilon>0$ and a finite subset $J \subset \N$ define a neighborhood $U_{J,\ep}$ of the identity  in $\hcA=(\R/\Z)^\N$ by 
\[
U_{J,\ep}=\{v=(v_j)\in(\R/\Z)^\N\,|\,|v_j|<\ep\text{ for all }j\in J\}
\]
where $|x+\Z| := \min_{k\in \Z} |x+k|$. The sets $U_{J,\ep}$ form a neighborhood basis of the identity in $\hcA$.

We will obtain the theorem from a sequence of lemmas.
\begin{lem}\label{lem:cont1}
For any $m>\ep^{-1}$ and $J \subset \N$, $\Ann(m\A)+U_{J,\ep}=\hcA$.
\end{lem}
\begin{proof}
Recall that $\Ann(m\A)=(\Z[1/m]/\Z)^\N$. In particular, both $\Ann(m\cA)$ and $U_{J,\ep}$ are product sets. So it suffices to show that for any $j \in J$ the projection of $\Ann(m\cA)+U_{J,\ep}$ onto the $j$-th coordinate is $\R/\Z$. Indeed, this projection is $\{k/m+\ep'+\Z:~0\leq k<m, ~|\ep'|<\ep\}$. If $m>\ep^{-1}$ then this set is all of $\R/\Z$.
\end{proof}
The first item now follows. Indeed, suppose that $f$ is a continuous function on $\Sub(\hcA)$. Then, by continuity, for each $\delta$ there are $J$ and $\ep$ such that $|f(\hcA)-f(H)|<\delta$ for each $H\in N_{U_{J,\ep}}(\hcA)$. Note that if $n_i>\ep^{-1}$ and $h\in \Hom(F_i,\hcA)$ we have that $\Ann(n_i\A)+h(F_i)\in N_{U_{J,\ep}}(\hcA)$ by the Lemma \ref{lem:cont1}. Thus
\[
\left|\lambda_i(f)-f(\hcA)\right|\leq \int_{\Hom(F_i,\hcA)} \left|f(\Ann(n_i\A)+h(F_i))-f(\hcA)\right|d\chi_{F_i}(h)\leq \delta
\]
where $\chi_{F_i}$ denotes Haar probability measure on $\Hom(F_i,\hcA)$. Since $\delta,f$ are arbitrary, this implies that if $n_i\to\infty$ then $\lambda_i \to \delta_\hcA$ as required. To prove the second item, we need the next lemma.
\begin{lem}\label{lem:top-2}
Let $V_{m,J,\epsilon}$ be the set of all $h\in\Hom(\Z/m\Z,\hcA)$ such that $h(\Z/m\Z)+U_{J,\ep}=\hcA$. Then for any fixed finite set $J \subset \N$ and any $\epsilon$, 
$$\lim_{m\to\infty} \chi_{\Z/m\Z}(V_{m,J,\epsilon}) = 1$$
where $\chi_{\Z/m\Z}$ denotes Haar probability measure on $\Hom(\Z/m\Z,\hcA)$. 
\end{lem}
\begin{proof}
Note that the map $h\mapsto h(1)$ identifies $\Hom(\Z/m\Z,\hcA)$ with $\Ann(m\A)=(\Z[1/m]/\Z)^\N$. For $j\in J$ define $v_j=h(1)_j\in \Z[1/m]/\Z$. Then $h(\Z/m\Z)+U_{J,\ep}=\hcA$ if and only if for each $j\in J$ 
$$\R/\Z = \{kv_j+\ep'+\Z:~0\leq k<m,~|\ep'|<\ep\}.$$
This in turn can happen only if $v_j\in \Z[1/d]/\Z$ for some $d > \ep^{-1}$. To show the statement of the lemma, it suffices to show that the ratio of all such $v_j$ in $\Z[1/m]/\Z$ goes to $1$ as $m\to\infty$. Or, what is the same, that the number of $v\in \Z[1/d]/\Z$, such that $d\leq \ep^{-1}$, divided by $m$, goes to $0$, but this is obvious.  
\end{proof}
From this lemma we have the second item. Indeed, given a continuous function $f$ on $\Sub(\hcA)$ and $\delta>0$ choose $J,\ep$ such that  $|f(\hcA)-f(H)|<\delta$ for each $H\in N_{U_{J,\ep}}(\hcA)$. Let $m_i$ be the maximum order of an element in $F_i$. Let 
$$Q_{i,J,\ep} = \{h \in \Hom(F_i,\hcA):~\Ann(n_i\A)+h(F_i)+U_{J,\ep}=\hcA.$$
Then $\chi_{F_i}(Q_{i,J,\ep}) \ge \chi_{F_i}(V_{m_i,J,\epsilon})$. Thus
\begin{equation*}
\begin{aligned}
&\left|\lambda_i(f)-f(\hcA)\right|\leq \int_{\Hom(F_i,\hcA)} \left|f(\Ann(n_i\A)+h(F_i))-f(\hcA)\right|~d\chi_{F_i}(h)\leq\\
& \int_{Q_{i,J,\ep}} \left|f(\Ann(n_i\A)+h(F_i))-f(\hcA)\right|~d\chi_{F_i}(h)+\\
&\int_{\Hom(F_i,\hcA)\setminus Q_{i,J,\ep}} \left|f(\Ann(n_i\A)+h(F_i))-f(\hcA)\right|~d\chi_{F_i}(h)\\
& \leq \delta + (1-\chi_{F_i}(Q_{i,J,\ep})) \|f\|_\infty \le \delta + (1-\chi_{\Z/m_i\Z}(V_{m_i,J,\ep}))\|f\|_\infty.
\end{aligned}
\end{equation*}
Lemma \ref{lem:top-2} now implies the second item.

To prove the third item, we need the next two lemmas.
\begin{lem}\label{lem:ratio00}
Fix $m >1$. Consider the set \[\{(a_1,\dots,a_k)\in(\Z[1/m]/\Z)^k\,|\,a_1,\dots, a_k\text{ generate }\Z[1/m]/\Z\}.\] Then the number of elements in this set divided by $m^k$ goes to $1$ as $k$ goes to infinity.
\end{lem}
\begin{proof}
It suffices to show that the number of elements in the complement of this set divided by $m^k$ goes to $0$. Note that the complement is the union, over $d|m$, $d\neq m$, of sets $(\Z[1/d]/\Z)^k$. Clearly $d^k/m^k\to 0$ as $k\to\infty$.

\end{proof}
\begin{lem}\label{lem:cont3}
Consider the set of those $h\in \Hom((\Z/m\Z)^k,\hcA)$ such that $h((\Z/m\Z)^k)+U_{J,\ep}\supset\Ann(m\A)$. The Haar measure of this set goes to $1$ as $k$ goes to infinity (for fixed $J,\ep,m$).
\end{lem}
\begin{proof}
Denote by $e_s$, $1\leq s\leq k$ the standard generators of $(\Z/m\Z)^k$. Since $J$ is finite, and projections onto coordinates of $\hcA=(\R/\Z)^\N$ are independent, it suffices to show that for each $j\in J$ 
$$\chi_{(\Z/m\Z)^k}\left(\left\{h \in \Hom( (\Z/m\Z)^k, \hcA):~ \textrm{Proj}_j( \textrm{Image}(h)) = \Z[1/m]/\Z\right\}\right) \to 1$$
as $k\to\infty$ where $\textrm{Proj}_j:(\R/\Z)^\N \to \R/\Z$ denotes projection onto the $j$-th coordinate. The number on the left is equal to the cardinality of the set \[\{(a_1,\dots,a_k)\in(\Z[1/m]/\Z)^k\,|\,a_1,\dots, a_k\text{ generate }\Z[1/m]/\Z\}.\] divided by $m^k$. By the Lemma \ref{lem:ratio00} we are done. 
\end{proof}

We will now prove the third item. Fix a finite set $J \subset \N$ and $\ep>0$. Recall that $F_i\cong F\oplus (\oplus_{j=1}^s(\Z/q_j^{t_j}\Z)^{m_j(i)})$, and let $k(i)=\min_j\{m_j(i)\}$, $F'_i=\oplus_{j=1}^s(\Z/q_j^{t_j}\Z)^{m_j(i)}$ and $F''_i=\oplus_{j=1}^s(\Z/q_j^{t_j}\Z)^{k(i)}=(\Z/m\Z)^{k(i)}$, where $m=\prod_j q_j^{t_j}$. By assumption, $k(i)\to\infty$ as $i\to\infty$. Note that $F''_i\subset F'_i$. Thus by Lemma \ref{lem:cont3} we have that the set $Q_i$ of those $h\in \Hom(F'_i,\hcA)$ such that $h(F'_i)+U_{J,\ep}\supset\Ann(m\A)$ has Haar measure going to $1$ as $i\to\infty$. In other words $Q_i$ is the set of those $h$ such that $h(F'_i)\in N_{U_{J,\ep}}(\Ann(m\A))$. We also have that for $h\in Q_i$ and for any $h_1\in\Hom(F,\hcA)$, $\Ann(n_i\A)+h(F'_i)+h_1(F)\in N_{U_{J,\ep}}(\Ann(n_i\A)+\Ann(m\A)+h_1(F))$.

We have now, for a continuous function $f$ on $\Sub(\hcA)$, using Fubini's Theorem
\begin{equation*}
\begin{aligned}
& \lambda_i(f)=\int_{\Hom(F,\hcA)}\int_{\Hom(F'_i),\hcA} f(\Ann(n_i\A)+h(F'_i)+h_1(F))\,d\chi_{F'_i}(h)\,d\chi_{F}(h_1).\\
\end{aligned}
\end{equation*}
It now follows, as in the proof of the second item, that the interior integral converges to $f(\Ann(n_i\A)+\Ann(m\A)+h_1(F))$, and so $\lambda_i(f)$ goes to 
\[
\int_{\Hom(F,\hcA)}f(\Ann(n_i\A)+\Ann(m\A)+h_1(F))\,d\chi_{F}(h_1)
\]
as $i\to\infty$. Because $f$ is arbitrary, this implies the third item.

The last, fourth, item is obvious.
\end{proof}

\subsection{Infinite product of Galois fields}\label{GO1}
In the last two subsections we derive by different means some results of \cite{GO09} on Grassmanians over a finite field.
Let $V$ be a locally compact vector space over $\F_q$. Endow  $\Aut_{\F_q}(V)$ with the Braconnier topology.  In this setup we use the duality of vector spaces instead of Pontryagin duality: if $V$ is a locally compact vector space over $\F_q$, then $V^*=\Hom_{\F_q}(V,\F_q)$, the set of continuous linear maps, endowed with the compact-open topology.

Since $\Aut_{\F_q}(V)\leq \Aut(V)$, Proposition \ref{prop:Pontryagin} applies to show that $\Aut_{\F_q}(V)\simeq \Aut_{\F_q}(V^*)$.
Note that $V^{**}\simeq V$ and $((\F_q)^\N)^*\simeq \oplus_\N \F_q$.

\begin{defn}
We call $\phi$ in  $\Aut_{\F_q}((\F_q)^\N)$ (respectively in $\Aut_{\F_q}(\oplus_\N \F_q)$) {\em finitary} if it changes only a finite number of coordinates. It is clear that finitary automorphisms form a group. Denote it by $\Aut_{\F_q}^{\textrm{fin}}((\F_q)^\N)$ (respectively $\Aut_{\F_q}^{\textrm{fin}}(\oplus_\N \F_q)$).
\end{defn}
It is easy to check that the isomorphism $\phi\mapsto (\phi^*)^{-1}$ between $\Aut_{\F_q}((\F_q)^\N)$ and $\Aut_{\F_q}(\oplus_\N \F_q)$ induces an isomorphism $\Aut_{\F_q}^{\textrm{fin}}((\F_q)^\N)\simeq \Aut_{\F_q}^{\textrm{fin}}(\oplus_\N \F_q).$

\begin{prop}\label{GO}
The subgroup of finitary automorphisms is dense.
\end{prop}
\begin{proof}
Since $(\F_q)^\N$ is compact and $\oplus_\N \F_q$ is discrete, the Braconnier topology in both cases coincides with the compact-open topology. 

It suffices to show that $\Aut_{\F_q}^{\textrm{fin}}(\oplus_\N \F_q)$ is dense in $\Aut_{\F_q}(\oplus_\N \F_q)$. Since $\oplus_\N \F_q$ is discrete, the compact-open topology coincides with the topology of pointwise convergence. Take any $\phi\in \Aut_{\F_q}(\oplus_\N \F_q)$ and finite subset $F \subset \oplus_\N \F_q$. It suffices to show there exists  $\psi\in\Aut_{\F_q}^{\textrm{fin}}(\oplus_\N \F_q)$ with $\phi\resto F = \psi\resto F$.

Note $F \subset \oplus_1^n \F_q$ for some $n$. Let $e_1,\dots e_n$ be the standard basis of  $\oplus_1^n \F_q$. Then $\phi(e_1),\dots \phi(e_n)$ is contained in some $\oplus_1^m \F_q$ for some $m$.  Let $N=\max(n,m)$. It follows that there is $\psi\in \Aut_{\F_q}(\oplus_1^N \F_q)$, such that $\psi$ and $\phi$ coincide on $\oplus_1^n \F_q$ and therefore $\psi$ and $\phi$ coincide on $F$.
\end{proof}

\begin{cor}
$\CRS_{\F_q}(\oplus_\N \F_q)$ consists of all measures that are $\Aut_{\F_q}^{{\textrm{fin}}}(\oplus_\N \F_q)$-invariant.
\end{cor}
\begin{proof}
Follows from Propositions \ref{GO} and \ref{cont_eval}.
\end{proof}
This shows that elements of $\CRS_{\F_q}(\oplus_\N \F_q)$ are exactly the random spaces defined in \cite{GO09}.

%The following two theorems are equivalent. Then proof is the same as for Theorems \ref{thm:hcA} and \ref{thm:cAn} (replacing $\hat{G}$ by $G^*$). 

%\begin{thm}\label{thm:hcB}
% Let $\lambda \in \CRS_{\F_q}^e( (\F_q)^\N)$, $\lambda\neq \delta_{(\F_q)^\N}$. Then there is unique $m\in \Z_+$ such that the random subgroup $h(\F_q^m)$ has law $\lambda$, where a random homomorphism $h\in\Hom_{\F_q}(\F_q^m,(\F_q)^\N)$ is given by the Haar measure on the compact group $\Hom_{\F_q}(\F_q^m,(\F_q)^\N)$. 
 %\end{thm}

%In the next Theorem we give a characterisation of $\CRS_{\F_q}^e(\oplus_\N \F_q)$.
%\begin{thm}\label{thm:cBn}
% Let $\lambda \in \CRS_{\F_q}^e(\oplus_\N \F_q)$, $\lambda\neq \delta_{\oplus_\N \F_q}$. Then there is unique $m\in\Z_+$ such that the random subgroup $\Ker(h)$ has law $\lambda$, where a random homomorphism $h\in\Hom_{\F_q}(\oplus_\N \F_q,\F_q^m)$ is given by the Haar measure on the compact group $\Hom_{\F_q}(\oplus_\N \F_q,\F_q^m)$. 
 % \end{thm}
The proofs of Theorems \ref{thm:hcB} and \ref{thm:cBn} proceed in the same way as the proof of Theorems \ref{thm:hcA} and \ref{thm:cAn} above, using the following simplified versions of the main auxiliary results.

The next result corresponds with Theorem \ref{thm:key2}:
\begin{thm}\label{t1}
Let $V$ be a countable vector space over $\F_q$ and let $\chi$ be an $\Aut_{\F_q}((\F_q)^\N)$-invariant and indecomposable Borel probability measure on $\Hom_{\F_q}(V,(\F_q)^\N)$. Then  there exist a subspace $W\le V$ such that $\chi$ is the normalized Haar measure on $\Hom_{\F_q}(V/W,(\F_q)^\N)$.
\end{thm}

%\begin{defn}
%Define $\overline{\Image}:\Hom_{\F_q}(G,(\F_q)^\N) \to \Subp((\F_q)^\N)$ by: $\overline{\Image}(h)$ is the closure of the image of $h$.
%\end{defn}
%We need the following proposition.
%We will prove this in the next section.
%Note that if $V$ be a countable vector space over $\F_q$, then either $V$ is finite or $V=\oplus_\N \F_q$. 

The next result corresponds with Proposition \ref{prop:haar}:
\begin{prop}\label{p1}
Let $V$ be a countable vector space over $\F_q$, $\chi$ denote the Haar measure on $\Hom_{\F_q}(V,(\F_q)^\N)$ and $h\in \Hom_{\F_q}(V,(\F_q)^\N)$ be a random homomorphism with law $\chi$. %For $h\in \Hom(G,\cA)$, let $\overline{\Image}(h) \in \Sub(\hcA)$ denote the closure of the image of $h$. % and $\overline{\Image}:\Hom(G,\hcA) \to \Sub(\hcA)$ the map $\overline{\Image}(h)$ is the closure of the image of $h$. Then
\begin{itemize}
\item If $V$ is finite then $h$ is injective a.s.
\item If $V=\oplus_\N \F_q$ then $\overline{\Image}(h) = (\F_q)^\N$ a.s. 
\end{itemize}
Note that if $V$ is a countable vector space over $\F_q$, then either $V$ is finite or $V=\oplus_\N \F_q$.
\end{prop}
The proofs of Theorem \ref{t1} and Proposition \ref{p1} are similar to the proofs of Theorem \ref{thm:key2} and Proposition \ref{prop:haar}.

%We call a linear invertible operator on $\oplus_\N \F_q$ finitary, if it changes only a finite number of coordinates.
%\begin{prop}
%Any measure on $\Subp(\oplus_\N \F_q)$, which is invariant with respect to all finitary operators, is in fact invariant with respect to all linear continuous invertible operators.
%\end{prop}
%This follows from the following
%\begin{lem}
%The group of all finitary linear invertible operators is dense in the group of all linear continuous invertible operators
%\end{lem}

\subsection{Finite dimensional approximation}\label{GO2}
%In this subsection we will consider the space $\Subp(\oplus_\N \F_q)$, where $q$ is a power of prime. 

Let $F_\kappa=(\F_q)^\kappa$, and denote by $\mu_\kappa$ the measure on $\Subp(\oplus_\N \F_q)$ that is the law of $\Ker(h)$, where $h\in\Hom_{\F_q}(\oplus_\N \F_q,F_\kappa)$ is a random homomorphism with law equal to Haar measure (see Theorem \ref{thm:cBn}). Note $\mu_\kappa$ is  $\Aut_{\F_q}(\oplus_\N \F_q)$-invariant. Let $V_n=\oplus_1^n \F_q\subset \oplus_\N \F_q$.

Consider the map $\Subp(\oplus_\N \F_q)\to\Subp(\oplus_\N \F_q)$ given by $X\mapsto X\cap V_n$. Let $\mu_{\kappa,n}$ denote the probability measure on $\Subp(\oplus_\N \F_q)$ which is the pushforward of $\mu_\kappa$ under this map. Note $\mu_{\kappa,n}$ is supported on $\Subp(V_n)$. Because $\mu_\kappa$ is  $\Aut_{\F_q}(\oplus_\N \F_q)$-invariant, $\mu_{\kappa,n}$ is  $\Aut_{\F_q}(V_n)$-invariant.

We are interested in the following numbers: $$\tilde{v}_{n,k}:=\mu_\kappa(\{X\in \Subp(\oplus_\N\F_q)\,|\, \dim_{\F_q}(X\cap V_n)=k\}),$$ $$v_{n,k}:=\mu_{\kappa,n}(\{W\}),\quad W\subseteq V_n,\quad \dim_{\F_q} W=k.$$
Because $\mu_{\kappa,n}$ is  $\Aut_{\F_q}(V_n)$-invariant, $v_{n,k}$ does not depend on a specific choice of subspace $W$. Applying  $\Aut_{\F_q}(V_n)$-invariance again, we obtain:
\begin{lem}
$\tilde{v}_{n,k}=d_{n,k}v_{n,k},$ where $d_{n,k}$ is the number of $k$-dimensional subspaces in $\F_q^n$.
\end{lem}
%\begin{thm}\label{appro}
%Suppose that $n\geq 0,$ $n\geq k\geq 0,$ and $\kappa\geq n-k$. Then $\tilde{v}_{n,k}$ is equal to the number of matrices of size $\kappa\times n$ which have rank $n-k$ divided by the number of all matrices of size $\kappa\times n$. Otherwise $\tilde{v}_{n,k}=0$.
%\end{thm}
\begin{proof}[Proof of Theorem \ref{appro}]
Let $\chi$ denote the Haar probability measure on $\Hom_{\F_q}(\oplus_\N \F_q,\F_q^\kappa)$. By definition $\mu_\kappa = \Ker_*\chi$. Thus
$$\tilde{v}_{n,k} = \Ker_*\chi(\{h \in \Hom_{\F_q}(\oplus_\N \F_q,\F_q^\kappa):~\dim(\Ker(h \resto V_n)) = k\}).$$
The restriction map $R:\Hom_{\F_q}(\oplus_\N \F_q,\F_q^\kappa) \to \Hom_{\F_q}(V_n,\F_q^\kappa)$ defined by $R(h) = h\resto V_n$ is a surjective homomorphism and therefore takes Haar measure to Haar measure. So if $\chi'$ is the Haar measure on $\Hom_{\F_q}(V_n,\F_q^\kappa)$ then
$$\tilde{v}_{n,k} = \Ker_*\chi'(\{h \in \Hom_{\F_q}(V_n,\F_q^\kappa):~\dim(\Ker(h)) = k\}).$$
We observe that $\Hom_{\F_q}(V_n,\F_q^\kappa)$ is a finite set which may be identified with the set of all $\kappa\times n$ matrices with values in $\F_q$. So $\chi'$ is the uniform probability measure.  The Theorem is finished by observing that  $h \in \Hom_{\F_q}(V_n,\F_q^\kappa)$ satisfies $\dim(\Ker(h))=k$ if and only if the matrix representing $h$ has rank $n-k$.
\end{proof}

We define two quantities.
\[
t_n:=(q^n-1)(q^n-q)\cdots(q^n-q^{n-1}),
\]
\[
s_n:=(q^n-1)(q^{n-1}-1)\cdots(q-1).
\]

\begin{cor}\label{computation}
If $n\geq 0,$ $n\geq k\geq 0,$ and $\kappa\geq n-k$ then
\begin{equation*}
 \tilde{v}_{n,k}=q^{(n-k)(n-k-1)/2-\kappa n}\frac{s_\kappa s_n}{s_{(n-k)}s_{(\kappa-n+k)}s_{k}}
 \end{equation*}
 \end{cor}
\begin{proof}
Recall that 
\begin{equation*}
\begin{aligned}
& |GL_n(q)|=t_n=q^{\binom{n}{2}}s_n.
\end{aligned}
\end{equation*}
To see this, observe that $GL_n(q)$ is in 1-1 correspondence with the set of all $n$-tuples $(v_1,\ldots, v_n)$ such that each $v_i \in \F_q^n$, $v_1\ne 0$ and for each $i<n$, $v_{i+1}$ is not in the span of $v_1,\ldots, v_i$. Hence there are $q^n-1$ choices for $v_1$, $q^{n}-q$ choices for $v_2$, $q^n-q^2$ choices for $v_3$ and so on.

The group $GL_\kappa(q)\times GL_n(q)$ acts on the set $M_{\kappa,n}(\F_q)$ of all matrices of size $\kappa\times n$ by the rule $(A,B)M=AMB^{-1}$. The orbits of this action are exactly the sets of matrices of constant rank. 
The matrix
$$
M_r:=\left(\begin{matrix}
 1_{r\times r} & 0_{r\times (n-r)} \\
  0_{(\kappa-r)\times r} & 0_{(\kappa-r)\times (n-r)}
 \end{matrix}\right)
 $$
 has rank $r$. The elements of $GL_\kappa(q)\times GL_n(q)$ that stabilize it have the form
 $$
\left(\left(\begin{matrix}
  a_{r\times r} & b_{r\times (\kappa-r)} \\
  0_{(\kappa-r)\times r} & d_{(\kappa-r)\times (\kappa-r)}
 \end{matrix}\right)
,
\left(\begin{matrix}
  a_{r\times r} & 0_{r\times (n-r)} \\
  c_{(n-r)\times r} & d'_{(n-r)\times (n-r)}
 \end{matrix}\right)^{-1}\right).
 $$
 Thus the number of matrices of rank $r$ in $M_{\kappa\times n}(\F_q)$ is
 \begin{equation*}
 \begin{aligned}
&\textrm{rank}_{\kappa,n,r}=\frac{|GL_\kappa(q)\times GL_n(q)|}{|\textrm{Stab}(M_r)|} = \frac{t_\kappa t_n}{t_r t_{\kappa-r}t_{n-r}q^{r(\kappa-r)+r(n-r)}}=\\
&q^{\binom{\kappa}{2}+\binom{n}{2}-\binom{r}{2}-\binom{\kappa-r}{2}-\binom{n-r}{2}-r(\kappa-r)-r(n-r)}\frac{s_\kappa s_n}{s_{r}s_{(\kappa-r)}s_{(n-r)}}=q^{r(r-1)/2}\frac{s_\kappa s_n}{s_{r}s_{(\kappa-r)}s_{(n-r)}}.
\end{aligned}
 \end{equation*}
 Hence
 \begin{equation*}
 \tilde{v}_{n,k}=\frac{\textrm{rank}_{\kappa,n,n-k}}{q^{\kappa n}}=q^{(n-k)(n-k-1)/2-\kappa n}\frac{s_\kappa s_n}{s_{(n-k)}s_{(\kappa-n+k)}s_{k}}.
 \end{equation*}
\end{proof}
Recall that $d_{n,k}$ is the number of $k$-dimensional subspaces in $\F_q^n$. We have
\begin{equation}
d_{n,k}=\frac{(q^n-1)(q^n-q)\cdots(q^n-q^{k-1})}{(q^k-1)(q^k-q)\cdots(q^k-q^{k-1})}=\frac{s_n}{s_{(n-k)}s_k}.
\end{equation}
To see this, observe $d_{n,k}$ is the number $n$-tuples $(v_1,\ldots, v_k)$ with $v_i \in \F_q^n$ such that the span of $v_1,\ldots, v_k$ is $k$-dimensional divided by $|GL_k(q)|$. Thus
\begin{cor} If $n\geq 0,$ $n\geq k\geq 0,$ and $\kappa\geq n-k$ then
\[
v_{n,k}=\frac{\tilde{v}_{n,k}}{d_{n,k}}=q^{(n-k)(n-k-1)/2-\kappa n}\frac{s_\kappa}{s_{(\kappa-n+k)}}.
\]
\end{cor}
%Now it is easy to check that
%\[
%v_{n,k}=v_{n+1,k}+q^{n-k}v_{n+1,k+1}.
%\]

\appendix
\section{An alternative proof of Theorem \ref{thm:key2} for the case $G=\Z$}

We present here an alternative proof of Theorem \ref{thm:key2} for the case $G=\Z$. It is then possible (and that was the case in the earlier version of the paper) to deduce the general case of Theorem \ref{thm:key2}, and thus to obtain another proof of Theorems \ref{thm:hcA}, \ref{thm:cAn}.
Note that since $\Hom(\Z,\hcA)=\hcA$, and so we will characterize all $\Aut(\hcA)$-invariant measures on $\hcA$.

Denote by $\nu_r$ the Haar measure on $\hat{\cA_r}=(\Z[1/r]/\Z)^\N$, $r\geq 1$. Denote by $\nu_\infty$ the Haar measure on $\hat{\cA}$. Since $\hat{\cA_r}$ is a subset of $\hcA$, we will also consider $\nu_r$ as measures on $\hcA$. 
\begin{thm}\label{T0}
 Any $\Aut(\hcA)$-invariant ergodic measure on $\hat{\cA}=\T^\N$ is equal to $\nu_r$ for some $r\in[1,\infty]$.
 \end{thm}
We will prove this by the end of the appendix. The proof avoids the de Finetti-Hewitt-Savage Theorem, using instead tools from abstract harmonic analysis.  

 By  $\nu_r^{(2)}$ we denote the Haar measure on $(\Z[1/r]/\Z)^2$, and by $\tau_r^{(2)}$ we denote the uniform measure on the set of pairs $(x,y)\in(\Z[1/r]/\Z)^2$ such that $x,y$ generate $\Z[1/r]/\Z$.  By $\nu^{(2)}_\infty$ we denote the Haar measure on $\T^2$. Once again, since $(\Z[1/r]/\Z)^2$ is a subset of $\T^2$, we also consider $\nu_r^{(2)}$ and $\tau_r^{(2)}$ as measures on $\T^2$.

 The plan of the proof of Theorem \ref{T0} is as follows: first we will show that any $\Aut(\T^2)$-invariant ergodic measure on $\T^2$ is either a Haar measure $\nu_\infty^{(2)}$ or some $\tau_r^{(2)}$. Then we will show that any $\Aut(\hcA)$-invariant ergodic measure on $\hcA=\T^\N$ is either the Haar measure $\nu_\infty$ on $\hcA$ or the Haar measure $\nu_r$ on $\hcA_r$ for some $r\geq 1$. We are going to prove it by projecting such a measure onto the first two coordinates and using the results obtained for $\T^2$.  However, when we project ergodic measures $\nu_r$ onto $\T^2$, we obtain the measures $\nu_r^{(2)}$ which are {\it not} ergodic. Thus we will have to establish the relationship between measures $\nu_r^{(2)}$ and $\tau_r^{(2)}$.

%$\cA=\oplus_{\N}\Z$ and $\cA_r=\oplus_{\N}\left(\Z/r\Z\right)$ (Note that $\cA_1$ is the trivial group). Let $\pi_r:\cA\to\cA_r$ be the projection.

% Suppose $\nu$ is a measure on $\hat{\cA_r}$, invariant under the action of all automorphisms of $\cA_r$ and ergodic. Then $(\hat{\pi}_r)_*\nu$ is a measure on $\hat{\cA}$, invariant under all automorphisms of $\cA$. So to find all characteristic measures on $\hat{\cA_r}$ it suffices to find all such measures on $\hat{\cA}$.

 \begin{thm}\label{T-1}
The Haar measure $\nu^{(2)}_\infty$ on the torus $\T^2$ is the only nonatomic $SL(2,\Z)$-invariant ergodic probability measure.  Any atomic $SL(2,\Z)$-invariant ergodic probability measure on $\T^2$ is $\tau_r^{(2)}$ for some $r\geq 1$.
\end{thm}
\begin{remark}
This theorem is very well-known but we were unable to find an elementary proof of it in the literature, which is why we provide such a proof below. For example, Theorem \ref{T-1} follows from Ratner's theorems, by inducing the action of $SL(2,\Z)$ on $\T^2$ to $SL(2,\R)$. The induced action is isomorphic to the (left) action of $SL(2,\R)$ on $(SL(2,\R)\ltimes\R^2)/(SL(2,\Z)\ltimes\Z^2)$ with an $SL(2,\R)$-invariant measure. The Ratner measure classification theorem then says that there is a closed subgroup $H$, such that $SL(2,\R)\subset H\subset SL(2,\R)\ltimes\R^2$ and the measure is the homogeneous measure on a closed orbit of $H$. There are only two such groups $H$, $H=SL(2,\R)\ltimes\R^2$ and $H=SL(2,\R)$. For the former, there is only one $H$-orbit, and it corresponds to the Haar measure $\nu^{(2)}_\infty$ on $\T^2$. For $H=SL(2,\R)$ it can be seen that the closed orbits are exactly $(SL(2,\R)\ltimes (\Z[1/r])^2)/(SL(2,\Z)\ltimes \Z^2)$, for $r\geq 1$, and they correspond to the measures $\tau^{(2)}_r$.
\end{remark}

\begin{proof}[Proof of Theorem \ref{T-1}]

Let $\nu$ be an ergodic $SL(2,\Z)$-invariant Borel probability measure on $\T^2$. Let $X$ be the set of all $(x,y) \in \T^2$ such that both $x$ and $y$ are irrational. Let $X' = \cap_{A \in SL(2,\Z) } AX$. Then $X'$ is $SL(2,\Z)$-invariant and therefore, $\nu(X') \in \{0,1\}$. Suppose for the moment that $\nu(X')=1$. Let 
\begin{displaymath}
N_1= \left[ \begin{array}{cc}
1 & 1 \\
0 & 1 
\end{array}\right].\end{displaymath}
We observe that for any $(x,y) \in X'$, $N_1(x,y) = (x+y,y)$. Denote by $q,\,p: \T^2 \to \T$ the maps $q(x,y)=x$ and $p(x,y)=y$. Then we have the decompositions
\begin{equation}\label{p3}
\nu=\int\nu_y~d\,p_*\nu(y)\;\text{ and }\;\nu=\int\nu^x~d\,q_*\nu(x),
\end{equation}
 where $\nu_y$ is the fiber measure on $\T\times\{y\}$ and $\nu^x$ is the fiber measure on $\{x\}\times\T$. Since $(N_1)_*\nu=\nu$, the measures $\nu_y$ must be invariant under addition by $y$. Because $\nu$ almost every point is doubly-irrational we have that $p_*\nu(\Q)=0$. Since irrational rotations of the circle are uniquely ergodic, this implies that $\nu_y$ is the Haar measure on $\T\times\{y\}$ (for $p_*\nu$-a.e. $y$). Similarly, using $N_2(x,y)=(x,x+y)$, we obtain that $\nu^x$ is the Haar measure on $\{x\}\times\T$ (for $q_*\nu$-a.e. $x$).
 
  Now we have, by pushing forward by $p$ the second formula in \eqref{p3}
\begin{equation}\label{p4}
p_*\nu = \int p_*\nu^x~d\,q_*\nu(x).
\end{equation}
Note that $p:\{x\}\times\T\to\T$ is a bijection. Since $\nu^x$ is the Haar measure on $\{x\}\times\T$ (for $q_*\nu$-a.e. $x$) it follows that $p_*\nu^x$ is the Haar measure on $\T$ (for $q_*\nu$-a.e. $x$). Therefore we have that $p_*\nu$ is Haar measure on $\T$. This implies, by the first formula in \eqref{p3}, that $\nu$ is the Haar measure.

So we may now assume that $\nu(X')=0$. We may also assume, without loss of generality, that $\nu( (\Q/\Z)^2) ) =0$ since otherwise $\nu$ is supported on a finite orbit. 

So for $\nu$-almost every $(x,y)$ there exists $A \in SL(2,\Z)$ such that exactly one coordinate of $A(x,y)$ is rational. 
Let 
\begin{displaymath}
B= \left[ \begin{array}{cc}
0 & 1 \\
-1 & 0 
\end{array}\right].\end{displaymath}
Then either $A(x,y)$ or $BA(x,y)$ is contained in $\T \times (\Q/\Z)$. So if we let $Y$ be the set of all $(x,y) \in \T^2$ such that $x$ is irrational and $y$ is rational then $\nu$-almost every orbit intersects $Y$ nontrivially. 

Let $A \in SL(2,\Z)$. If $AY \cap Y \ne \emptyset$ then $A$ is in the subgroup $N$ generated by $N_1$ and $-I$. Moreover $NY=Y$. So if $T \subset SL(2,\Z)$ is a set of coset representatives for $N$ in $SL(2,\Z)$ then the sets $\{gY:~g\in T\}$ are pairwise disjoint and their union has full $\nu$-measure. So $1 = \sum_{g\in T} \nu(gY)$. Since $\nu$ is measure-preserving $\nu(gY)=\nu(Y)$. So $1 = |T| \nu(Y)$. However $T$ is infinite. This contradiction proves the first part of the statement of the theorem.

Suppose now that $\nu$ is an $SL(2,\Z)$-invariant ergodic probability measure on $\T^2$, and $\nu(\{(x,y)\})>0$. Then the $SL(2,\Z)$-orbit of $(x,y)$ must be finite and since $\nu$ is ergodic, $\nu$ is the uniform measure on this orbit. 

In particular, elements $(x+ny,y)$ and $(x,mx+y)$ belong to this orbit for all $n,m\in\Z$. Since this orbit is finite, it follows that $(n-n')y\in\Z$ for some $n\neq n'$, and so $y$ is rational. Similarly, $x$ is rational.

Let $r\geq 1$ be such that $x,y$ generate $\Z[1/r]/\Z$. Then $(x,y)\in (\Z[1/r]/\Z)^2$, and the $SL(2,\Z)$ orbit of $(x,y)$ is the set of all $(z,t)\in (\Z[1/r]/\Z)^2$ such that $z,t$ generate $\Z[1/r]/\Z$.
\end{proof}

Using Choquet theorem what we have just proven can be reformulated as follows: if $\nu$ is an $\Aut(\T^2)$-invariant probability measure then it is the baricenter of some probability measure on the set $\{\nu^{(2)}_\infty\}\cup \{\tau^{(2)}_r\}_{r\geq1}$.  In other words, there are nonnegative numbers $b_r$, $r\in[1,\infty]$ such that $\sum_{r\in[1,\infty]} b_r=1$ and $\nu=b_\infty \nu^{(2)}_\infty+\sum_{r\in [1,\infty)}b_r\tau^{(2)}_r$. However, we will need to replace $\tau^{(2)}_r$ by $\nu^{(2)}_r$. This is the content on the next three lemmas.

We start by expressing $\tau_r^{(2)}$ as a linear combination of $\nu_r^{(2)}$. Recall that the Mobius function $\phi:\N\to\R$ is defined by $\phi(n)=0$ if $p^2|n$ for some prime $p$, and $\phi(p_1\cdots p_t)=(-1)^t$ if $p_1,\dots,p_t$ are different prime numbers. We will use that $\phi$ is {\it multiplicative}, that is if $r_1,r_2$ are mutually prime, then $\phi(r_1r_2)=\phi(r_1)\phi(r_2)$. We will also use the {\it Mobius inversion formula} which can be stated as follows: if we have two sequences $c_r$ and $d_r$, and $c_r=\sum_{s|r}d_s$, then $d_r=\sum_{s|r}\phi(r/s)c_s$.

 \begin{lem}\label{LA0}
Let $\phi$ denote the Mobius function. Then
\begin{equation} 
\tau_r^{(2)}=\sum_{k|r}\alpha(k,r)\nu_k^{(2)},
\end{equation}
where
\begin{equation}
\alpha{(k,r)}=\frac{\phi\left(r/k\right)k^2}{\sum_{s|r}\phi\left(r/s\right)s^2}=\frac{\phi\left(r/k\right)k^2}{\beta(r)}
\end{equation}
\end{lem}
\begin{proof}
Let $\beta(r)$ be the number of elements in the set 
\[
S_r=\{(x,y)\in(\Z[1/r]/\Z)^2:\,x\text{ and }y\text{ generate }\Z[1/r]/\Z\}.
\]
Then $(\Z[1/r]/\Z)^2$ is the disjoint union of $S_k$, $k|r$. Thus $r^2=\sum_{k|r}\beta(k)$ and so, by the Mobius inversion formula  $\beta(r)=\sum_{k|r}\phi(r/k)k^2$.
We have that $\nu_r^{(2)}$ is the uniform measure on $(\Z[1/r]/\Z)^2$, and $\tau_k^{(2)}$ is the uniform measure on $S_k$. Thus
\[
r^2\nu_r^{(2)}=\sum_{k|r}\beta(k) \tau_k^{(2)}.
\] 
Again by the Mobius inversion formula 
\[
\beta(r) \tau_r^{(2)}=\sum_{k|r}\phi(r/k)k^2\nu_k^{(2)}.
\]
\end{proof}
\begin{lem}\label{LA2a}
$\beta(r)\geq cr^2$ for some $c>0$.
\end{lem}
\begin{proof}
It is easy to see that $\beta(r)$ is a multiplicative function of $r$. Also, $\beta(p^n)=p^{2n}-p^{2(n-1)}$. Let $p_1,\dots,p_t$ be all prime divisors of $r$. Then
\[
\frac{\beta(r)}{r^2}=\left(1-\frac{1}{p_1^2}\right)\cdots\left(1-\frac{1}{p_t^2}\right).
\] 
Thus
\[
\frac{\beta(r)}{r^2}\geq\prod_{p-\text{prime}}\left(1-\frac{1}{p^2}\right)=c>0
\]
\end{proof}
\begin{lem}\label{LA2b}
Suppose $\mu$ is an $\Aut(\T^2)$-invariant probability measure on $\T^2$. Then there exist a sequence $\{a_r\}_{r=1}^{r=\infty}$ such that $\sum_r|a_r|<\infty$ and 
\[
\sum_{r\in[1,\infty]}a_r \nu_r^{(2)}=\mu.
\]
\end{lem}
\begin{proof}
By theorem \ref{T-1} we have that $\{\nu_\infty^{(2)}\}\cup\{\tau_r^{(2)}\}_{r\geq 1}$ are the extreme points of the set of all $\Aut(\T^2)$-invariant measures on $\T^2$. Thus by Choquet theorem there is a probability measure on $\{\nu_\infty^{(2)}\}\cup\{\tau_r^{(2)}\}_{r\geq 1}$ such that $\mu$ is equal to its baricenter. That is, there is a sequence of numbers $\{b_r\}_{r=1}^{r=\infty}$ such that $b_r\geq 0$ and 
\[
\mu(U)=b_\infty \nu^{(2)}_\infty+\sum_{r\in [1,\infty)}b_r\tau_r^{(2)}(U)\text{ for all measurable }U\subset \T^2.
\]
By lemma \ref{LA0} we have that 
 \[
 \mu-b_\infty \nu^{(2)}_\infty=\sum_{r\in[1,\infty)} b_r \left(\sum_{k|r}\alpha(k,r)\nu_k^{(2)}\right)=\sum_{k\in [1,\infty)}\sum_{q\geq 1}b_{kq}\alpha(k,kq)\nu_k^{(2)}.
 \]
 We have that
 \[
 \left|b_{kq}\alpha(k,kq)\right|=b_{kq}\frac{k^2}{\beta(kq)}\leq c^{-1}\frac{b_{kq}}{q^2}.
 \]
 Thus we have (using lemma \ref{LA2a}) 
 \[
 \sum_{k,q}\left|b_{kq}\alpha(k,kq)\right|\leq c^{-1}\sum_{k,q}\frac{b_{kq}}{q^2}\leq c^{-1}\sum_{q}\frac{1}{q^2}<\infty.
 \]
So to finish the proof it suffices to take $a_\infty=b_\infty$ and $a_r=\sum_{q\geq 1}b_{rq}\alpha(r,rq)$.
\end{proof}

\begin{lem}\label{LA1}
Let $m_i$, $i\geq 1$ be probability measures on a measurable space $X$. Suppose that $\sum_{i\geq 1}|a_i|<\infty$. Define $m$, a (signed) measure on $X$, by the rule
\[
m(U)=\sum_{i\geq 1}a_im_i(U),
\]
for any measurable set $U\subset X$. Then $\sum_{i=1}^Na_im_i$ converges  to $m$, as $N\to\infty$, in norm, and therefore $weak^*$. Conversely, if $\sum_{i=1}^Na_im_i$ $weak^*$ converges to $m'$, then $m=m'$.
\end{lem}
\begin{proof}
 Trivial. We note that one has to prove also that $m$ is indeed countably additive, but that is easy using $\sum_{i\geq 1}|a_i|<\infty$.
 \end{proof}
\begin{lem}\label{LA2}
Let $m$ and $m_i$, $i\geq 1$ be probability measures on $\hat{B}$, the Pontryagin dual of a discrete group $B$. Then $m_i$ $weak^*$ converges to $m$ iff $\hat{m}_i(b)\to \hat{m}(b)$ for all $b\in B$, where $\hat{m}$ is the Fourier-Stiltjes transform of $m$.
 \end{lem}
 \begin{proof}
 Recall that $\hat{m}(b)=\int_{\hat{B}}\chi(b)\,d\,\mu(\chi)$. That is, $b$ is considered as function on $\hat{B}$, $b(\chi)=\chi(b)$. This function is continuous, so $weak^*$ convergence $m_i\to m$ impies $\hat{m}_i(b)\to \hat{m}(b)$ for all $b\in B$.
 
 Conversely, any continuous function on the compact space $\hat{B}$ can be approximated by a finite linear combination of $b\in B$ (by Stone-Weierstrass theorem). 
 \end{proof}

 \begin{lem}\label{LA3}
 Let $\mu$ be a probability measure on $\hat{\cA}=\T^\N$, and let $\hat{\mu}:\oplus_{\N} \Z\to\C$ be the Fourier-Stiltjes transform of $\mu$. Then $\mu$ is invariant under automorphisms of $\cA$ if and only if $\hat{\mu}(k)=\hat{\mu}(\psi(k))$ for any $k\in\oplus_\N \Z$ and any automorphism $\psi$ of $\cA$. 
 \end{lem}
 \begin{proof}
 First note that in general, if $\psi:B_1\to B_2$ is a continuous homomorphism of locally compact groups, and $\mu$ is a finite measure on $\hat{B}_2$, then $\widehat{\hat{\psi}_*\mu}=\hat{\mu}\psi$. Indeed, we have
 \begin{equation}\label{property2}
 \widehat{\hat{\psi}_*\mu}(b_1)=\int_{\hat{B}_1}\chi(b_1)\,d\,\hat{\psi}_*\mu(\chi)=\int_{\hat{B}_2}\hat{\psi}(\zeta)(b_1)\,d\,\mu(\zeta)=\int_{\hat{B}_2}\zeta(\psi(b_1))\,d\,\mu(\zeta)=\hat{\mu}(\psi(b_1))
 \end{equation}
 
 It follows that if $\mu$ is an $\Aut(\hcA)$-invariant measure on $\hat{\cA}$ and $\psi\in \Aut(\A)$, then $\hat{\mu}=\widehat{\hat{\psi}_*\mu}=\hat{\mu}\psi$ by \eqref{property2}. Conversely, if $\hat{\mu}(k)=\hat{\mu}(\psi(k))$ for any $k\in\oplus_\N \Z$, then it follows by \eqref{property2} that $\hat{\mu}=\widehat{\hat{\psi}_*\mu}$, and thus $\mu=\hat{\psi}_*\mu$. Since this is true for any $\psi\in\Aut(\A)$, it follows that $\mu$ is $\Aut(\hcA)$-invariant.
 \end{proof}
 
 In the next Lemma and its Corollary we establish that it is possible to write $\hcA$ as a disjoint union of countable sets $\hcA'_r$ such that $\nu_r(\hcA'_r)=1$. 
 \begin{lem}\label{LA-3}
$\nu_r(\hcA_{r'})=1$ if $r|r'$ and $0$ otherwise.
\end{lem}
 \begin{proof}
 Note first that if $r|r'$ then $\hcA_{r}\subset\hcA_{r'}$. Thus $\nu_r(\hcA_{r'})\geq \nu_r(\hcA_{r})=1$.
 
 Suppose now that $r\not{|}r'$. Let $d=gcd(r,r')$. It follows that $d<r$. Since $\nu_r$ is supported on $\hcA_r$ we have that 
 \[
 \nu_r(\hcA_{r'})=\nu_r(\hcA_{r'}\cap \hcA_{r})=\nu_r(\hcA_{d}).
 \]
 It is left to notice that since $d<r$, the subgroup $\hcA_{d}$ has infinite index in $\hcA_{r}$. Since $\nu_r$ is a finite measure, it follows that $\nu_r(\hcA_{d})$ must be $0$.
 \end{proof}
 \begin{cor}\label{C-3}
 Denote by $\hcA'_s=\hcA_s-\cup_{k<s}\hcA_k$. Denote also $\hcA'_\infty=\hcA-\cup_k \hcA_k$. Then $\nu_r(\hcA'_s)=1$ if $r=s$ and $0$ otherwise.
 \end{cor}

 We are now ready to prove the main theorem of this section, about the description of $\Aut(\hcA)$-invariant ergodic measures on $\T^\N$.

 \begin{proof}[Proof of Theorem \ref{T0}]
 In order to show that $\{\nu_r\}_{r=1}^{r=\infty}$ is the list of all $\Aut(\hcA)$-invariant ergodic measures, by Choquet theorem it suffices to show that any $\Aut(\hcA)$-invariant probability measure is the baricenter of some probability measure on $\{\nu_r\}_{r=1}^{r=\infty}$. That is, if $\nu$ is an $\Aut(\hcA)$-invariant probability measure on $\hat{\cA}$ then there are numbers $a_r\geq 0$, $r\in[1,\infty]$ such that $\sum_r a_r \nu_r=\nu$.
 
 Consider the inclusion onto the first two coordinates
 \[
 i_2:\Z^2\to\A=\oplus_\N\Z.
 \]
 It is easy to show that 
 \[
 \hat{i}_2: \T^\N\to \T^2
 \]
 is the projection onto the first two coordinates. Thus we have that  $(\hat{i}_2)_*\nu$ is an $\Aut(\T^2)$-invariant measure on $\T^2$ and that $(\hat{i}_2)_*\nu_r=\nu_r^{(2)}$, the Haar measures on $(\Z[1/r]/\Z)^2$.
 
 By Lemma \ref{LA2b} we have that there exist a sequence $\{a_r\}_{r=1}^{r=\infty}$ such that $\sum_r|a_r|<\infty$ and 
\[
(\hat{i}_2)_*\nu=\sum_{r\in[1,\infty]}a_r \nu_r^{(2)}=\sum_{r\in[1,\infty]}a_r (\hat{i}_2)_*\nu_r.
\]
By Lemmas \ref{LA1} and \ref{LA2} it follows that 
\[
\widehat{(\hat{i}_2)_*\nu}=\sum_{r\in[1,\infty]}a_r\, \widehat{(\hat{i}_2)_*\nu_r}.
\]
 
Using formula \eqref{property2}, we obtain
 \[
 \widehat{(\hat{i}_2)_*\nu}=\hat{\nu}\,i_2,
  \]
  \[
 \widehat{(\hat{i}_2)_*\nu_r}=\hat{\nu_r}\,i_2.
  \]
  and thus 
  \[
  \hat{\nu}(k)=\sum_{r\in[1,\infty]}a_r\, \hat{\nu_r}(k),\text{ for all }k\in \Z^2=\Z\oplus \Z\oplus \{0\}\oplus \{0\}\oplus \dots.
  \]
  
 It is easy to show that any $\Aut(\hcA)$-orbit in $\oplus_\N\Z$ intersects $\Z^2$. By Lemma \ref{LA3} all $\hat{\nu}$, $\hat{\nu}_r$ are $\Aut(\A)$-invariant functions. Thus we have from the previous equality that 
  \[
  \hat{\nu}(k)=\sum_{r\in[1,\infty]}a_r\, \hat{\nu_r}(k),\text{ for all }k\in \Z^\N.
  \]
 Using Lemmas \ref{LA1} and \ref{LA2} again we obtain that 
 \[
 \nu=\sum_{r\in[1,\infty]}a_r\, \nu_r.
 \]
 It is only left to show that all $a_r\geq 0$. Indeed, by Corollary \ref{C-3} we have that 
 \[
 a_{r}= a_r \nu_r(\hcA'_{r})=\nu(\hcA'_{r})\geq 0.
 \]
 \end{proof}

\end{document}